\documentclass[11pt, reqno]{amsart}
\usepackage{amssymb, amstext, amscd, amsmath, amsthm}
\usepackage{color}
\usepackage{hyperref}

\usepackage{xy}
\xyoption{all}

\theoremstyle{plain}
\newtheorem{theorem}{Theorem}[section]
\newtheorem{thm}[theorem]{Theorem}
\newtheorem{cor}[theorem]{Corollary}
\newtheorem{prop}[theorem]{Proposition}
\newtheorem{lem}[theorem]{Lemma}
\newtheorem*{theorem*}{Theorem}
\theoremstyle{definition}
\newtheorem{rem}[theorem]{Remark}
\newtheorem{defn}[theorem]{Definition}

\newtheorem{eg}[theorem]{Example}

\newtheorem{obs}[theorem]{Observation}

\newcommand{\bF}{{\mathbb{F}}}

\newcommand{\bN}{{\mathbb{N}}}

\newcommand{\bT}{{\mathbb{T}}}

\newcommand{\bZ}{{\mathbb{Z}}}

\newcommand{\bm}{{\mathbf{m}}}
\newcommand{\bn}{{\mathbf{n}}}

\newcommand{\bg}{{\mathbf{g}}}

  \newcommand{\A}{{\mathcal{A}}}

  \newcommand{\F}{{\mathcal{F}}}
  \newcommand{\G}{{\mathcal{G}}}

  \newcommand{\J}{{\mathcal{J}}}
  \newcommand{\K}{{\mathcal{K}}}
\renewcommand{\L}{{\mathcal{L}}}

\renewcommand{\O}{{\mathcal{O}}}

  \newcommand{\Q}{{\mathcal{Q}}}

\newcommand{\fA}{{\mathfrak{A}}}

\newcommand{\fk}{{\mathfrak{k}}}
\newcommand{\fs}{{\mathfrak{s}}}
\newcommand{\ft}{{\mathfrak{t}}}

\newcommand{\rC}{{\mathrm{C}}}

\newcommand{\upchi}{{\raise.35ex\hbox{\ensuremath{\chi}}}}

\newcommand{\foral}{\text{ for all }}

\newcommand{\qforal}{\quad\text{for all}\quad}

\newcommand{\Aut}{\operatorname{Aut}}

\newcommand{\id}{{\operatorname{id}}}

\newcommand{\spn}{\operatorname{span}}

\newcommand{\ca}{\mathrm{C}^*}

\newcommand{\Fn}{\mathbb{F}_n^+}

\newcommand{\Fth}{\mathbb{F}_\theta^+}

\newcommand{\mt}{\varnothing}

\newcommand{\ol}{\overline}

\begin{document}
\title[Boundary quotient C*-algebras]{Boundary quotient C*-algebras of products of odometers}
\author[H. Li]{Hui Li}
\address{Hui Li,
Research Center for Operator Algebras and Shanghai Key Laboratory of Pure Mathematics and Mathematical Practice, Department of Mathematics, East China Normal University, 3663 Zhongshan North Road, Putuo District, Shanghai 200062, China}
\address{Department of Mathematics $\&$ Statistics, University of Windsor, Windsor, ON N9B 3P4, CANADA}
\email{lihui8605@hotmail.com}
\author[D. Yang]{Dilian Yang}
\address{Dilian Yang,
Department of Mathematics $\&$ Statistics, University of Windsor, Windsor, ON
N9B 3P4, CANADA}
\email{dyang@uwindsor.ca}

\thanks{The first author was partially supported by Research Center for Operator Algebras of East China Normal University and was partially supported by Science and Technology Commission of Shanghai Municipality (STCSM), grant No. 13dz2260400.}
\thanks{The second author was partially supported by an NSERC Discovery Grant 808235.}

\begin{abstract}
In this paper, we study the boundary quotient C*-algebras associated to products of odometers. One of our main results
shows that the boundary quotient C*-algebra of the standard product of $k$ odometers
over $n_i$-letter alphabets ($1\le i\le k$) is always nuclear, and that
it is a UCT Kirchberg algebra
if and only if $\{\ln n_i: 1\le i\le k\}$ is rationally independent,
if and only if the associated single-vertex $k$-graph C*-algebra is simple,
To achieve this, one of our main steps is to construct a topological $k$-graph such that
its associated Cuntz-Pimsner C*-algebra is isomorphic to the boundary quotient C*-algebra.
Some relations between the boundary quotient C*-algebra and the C*-algebra $\Q_\bN$ introduced by Cuntz are also
investigated.
As an easy consequence of our main results, it settles a boundary quotient C*-algebra
constructed by  Brownlowe-Ramagge-Robertson-Whittaker.
\end{abstract}


\maketitle

\section{Introduction}

In \cite{Li12}, Xin Li associated several C*-algebras to a discrete left cancellative semigroup $P$. One of them is called the full  C*-algebra $\ca(P)$ of $P$, which
is generated by an isometric representation of $P$ and a family of projections parametrized by a family of right ideals of $P$ satisfying certain relations.
Since then, there has been attracting a lot of attention to the study of semigroup C*-algebras. See, for example, \cite{ABLS16, BOS15, BLS17, BRRW14, Stam16, Star15} and the
references therein. In \cite{BRRW14}, Brownlowe-Ramagge-Robertson-Whittaker defined a quotient C*-algebra $\Q(P)$ of $\ca(P)$. They called it the boundary quotient
of $\ca(P)$. Roughly speaking, if we think of $\ca(P)$ as a ``Toeplitz type" C*-algebra, then $\Q(P)$ is of ``Cuntz-Pimsner type".

In \cite[Section 6]{BRRW14}, the authors investigated many examples of the boundary quotients of the full C*-algebras of semigroups coming from Zappa-Sz\'ep products,
which are also right least common multiple (LCM) semigroups. The last example there, i.e., \cite[Subsection 6.6]{BRRW14},
is concerned with the standard product of two odometers $(\bZ, \{0,1,\ldots, n-1\})$ and $(\bZ, \{0,1,\ldots, m-1\})$,
where $m$ and $n$ are two coprime positive integers greater than $1$. If we ``divide" the elements in $\{0,1,\ldots, mn-1\}$ by $n$ and $m$ respectively, then we get a bijection
$\theta$ from $\{0,1,\ldots, n-1\}\times \{0,1,\ldots, m-1\}$ to $\{0,1,\ldots, m-1\}\times \{0,1,\ldots, n-1\}$. Then we obtain a special semigroup $\Fth$, which is actually
a single-vertex $2$-graph (see \cite{DY091}). Since $n$ and $m$ are coprime, $\Fth$ is right LCM. Moreover, one can form the Zappa-Sz\'ep product
$\Fth\bowtie \bZ$, which turns out to be right LCM too, and so falls into the class studied in \cite{BRRW14}. It is easy to see that the $2$-graph C*-algebra  $\O_\theta$ of $\Fth$ is simple as
the coprimeness of $n$ and $m$ implies
the aperiodicity of $\Fth$ (see \cite{DY091}). However, unlike the other examples, $\ca(\Fth\bowtie \bZ)$ was not well-understood there.
As stated at the end of \cite{BRRW14}, ``It would be interesting, albeit outside the scope of this paper, to further understand $\ca(\Fth\bowtie\bZ)$ from the point of view of these
existing constructions."

Our original motivation of this paper comes from the example in \cite[Subsection 6.6]{BRRW14} mentioned above. One of the main results in this paper
shows that the boundary quotient C*-algebras $\Q(\Fth\bowtie \bZ)$ of the standard products $\Fth\bowtie \bZ$ of odometers $\{(\bZ, \{0,1,\ldots, n_i-1\})\}_{i=1}^k$
are always nuclear, but simple if and only if $\{\ln n_i: 1\le i\le k\}$ is rationally independent and also purely
infinite in these cases.
This obviously settles the boundary quotient C*-algebra introduced in \cite{BRRW14}.
To achieve our main goal, we first construct  a family of topological $k$-graphs,
which are $k$-dimensional analogues of Katsura's topological graph $E_{n,1}$ $(n\in \bN)$
in \cite{Kat08},
and then show that their associated Cuntz-Pimsner C*-algebras
are isomorphic to $\Q(\Fth\bowtie \bZ)$.
On the way to our main results, we carefully study the generators and relations of the boundary quotient C*-algebras of a class of
Zappa-Sz\'ep products of the form $\Fth\bowtie G$, where $\Fth$ is a single-vertex $k$-graph and $G$ is a group. We should mention that $\Fth$ here is
not necessary to be right LCM, and so $\Fth\bowtie G$ is not right LCM in general. Therefore, one can not apply the results in
the recent works on right LCM semigroups, such as \cite{ABLS16, BOS15, BLS17, BRRW14, Stam16, Star15}, to our cases.

The rest of this paper is organized as follows. In Section \ref{S:pre}, some necessary background, which will be used later, is given. With very careful analysis,
in Section \ref{S:genrel}, we exhibit the generators and relations of the boundary quotient C*-algebras of  Zappa-Sz\'ep products of the form $\Fth\bowtie G$,
where $\Fth$ is a single-vertex $k$-graph and $G$ is a group (see Theorem \ref{T:kQ(LambdaxG)}). As an important application of the results in Section \ref{S:genrel},
we obtain a very simple presentation of the boundary quotient C*-algebra of the standard product of $k$ odometers (see Definition \ref{D:podo}) in Section \ref{S:pssa}.
Roughly speaking, it is the universal C*-algebra generated by a unitary representation of $G$ and a $*$-representation of $\Fth$
 which are compatible with the odometer actions  (see Theorem \ref{T:universal}).
In our main section, Section \ref{S:main}, we first construct a class of topological $k$-graphs $\{\Lambda_\bn:\bn\in\bN^k\}$, which is a higher-dimensional analogue of a
class of topological graphs $\{E_{n,1}: n\in \bN\}$ given by Katsura in \cite{Kat08}. We associate to $\Lambda_\bn$ a product system $X(\Lambda_\bn)$ over $\bN^k$.
The first main result in this section shows that the associated Cuntz-Pimsner C*-algebra $\O_{X(\Lambda_\bn)}$
of $\Lambda_\bn$ is isomorphic to the boundary quotient C*-algebra of the standard product of $k$ odometers (see Theorem \ref{T:top2}).
Then, motivated by and with the aid of some results in \cite{Cun08, Kat08, Yam09}, we prove
Theorem \ref{T:simple} which states as follows:
$\O_{X(\Lambda_\bn)}$ is simple if and only if $\{\ln n_i: 1\le i\le k\}$ is rationally independent, and $\O_{X(\Lambda_\bn)}$ is also purely infinite in these cases.
The nuclearity of $\O_{X(\Lambda_\bn)}$ is obtained by applying some results in \cite{CLSV11, Yee07} to our case. 
Also $\O_{X(\Lambda_\bn)}$ satisfies the UCT from \cite{Tu99}. 
To summarize, we have the following results:
\begin{theorem*}[Theorem \ref{T:top2} and Theorem \ref{T:simplepure}]
Let $\Fth\bowtie\bZ$ be the Zappa-Sz\'ep product induced from the standard product of $k$ odometers $\{(\bZ, \{0,1,\ldots, n_i-1\})\}_{i=1}^k$. Then
\begin{itemize}
\item[(I)] $\Q(\Fth\bowtie \bZ)$ is isomorphic to $\O_{X(\Lambda_\bn)}$;

\item[(II)]
$\Q(\Fth\bowtie \bZ)$ is nuclear;

\item[(III)]
$\Q(\Fth\bowtie \bZ)$ is a unital UCT Kirchberg algebra $\Leftrightarrow$ $\{\ln {n_i}\}_{1\le i\le k}$ is rationally independent $\Leftrightarrow$ $\O_\theta$ is simple
$\Leftrightarrow$ $\Fth$ is aperiodic.
%
\end{itemize}
\end{theorem*}
Finally, some relations between  $\Q(\Fth\bowtie \bZ)$ and the C*-algebra $\Q_\bN$ introduced by Cuntz in \cite{Cun08} are also briefly discussed at the end of the paper.


%
%
%
%
%

\section{Preliminaries}

\label{S:pre}

In this section, we provide some necessary background which will be useful later. We also take this chance to fix our terminologies and notation.

\subsection*{Notation and Conventions}

Let $\mathbb{N}$ be the additive semigroup of non-negative integers. Denote by $\mathbb{N}^\times$ the multiplicative semigroup of positive integers. Let $1 \leq k \leq \infty$. For any semigroup $P$, denote by $P^k$ (resp.~$\prod_{i=1}^{k}P$) the direct sum (resp.~product) of $k$ copies of $P$ (they coincide if $k < \infty$). Let $\{e_i\}_{i=1}^{k}$ be the standard basis of $\mathbb{N}^k$. For $n \in \mathbb{N}^k$, we write $n=(n_1,\ldots, n_k)$. For $n,m \in \mathbb{N}^k,z \in \prod_{i=1}^{k}\mathbb{T}$,
denote by $n \lor m$ (resp.~$n\wedge m$) the coordinatewise maximum (resp.~minimum) of $n$ and $m$, and $z^n:=\prod_{i=1}^{k}z_i^{n_i}$.

For $1\le n\in \bN$, let $[n]:=\{0,1,\ldots, n-1\}$. By $\Fn$, we mean the unital free semigroup with $n$ generators.

In this paper, $k$ is an arbitrarily fixed positive integer which could also be $\infty$, unless otherwise specified.

All semigroups in this paper are assumed to be unital (and so are monoids). For a semigroup $U$, its identity is denoted by $1_U$ (or just $1$ if the context is clear).

\subsection{Cuntz-Pimsner algebras of product systems over $\mathbb{N}^k$}
In this subsection we recap the notion of product systems over $\mathbb{N}^k$ from \cite{Fow02}.

Let $A$ be a C*-algebra. A \emph{C*-correspondence} over $A$ (see \cite{Fow02, FMR03}) is a right Hilbert $A$-module $X$ together with a $*$-homomorphism $\phi: A\to \L(X)$, which gives a left
action of $A$ on $X$ by $a\cdot x:=\phi(a)x$ for all $a\in A$ and $x\in X$.  A (Toeplitz) \textit{representation} of $X$ in a C*-algebra $B$ is a pair $(\psi, \pi)$, where $\psi: X\to B$
is a linear map, and $\pi: A\to B$ is a homomorphism such that
\[
\psi(a\cdot x)=\pi(a)\psi (x);\  \psi(x)^*\psi(y)=\pi(\langle x,y\rangle)\qforal a\in A\text{ and }x,y\in X.
\]
Notice that the relation $\psi(x\cdot a)=\psi (x)\pi(a)$ holds automatically due to the above second relation.
It turns out that there is a homomorphism $\psi^{(1)}: \K(X)\to B$ satisfying
\begin{align}
\label{E:psi(1)}
\psi^{(1)}(\Theta_{x,y})=\psi(x)\psi(y)^*\qforal x,y\in X,
\end{align}
where $\Theta_{x,y}(z):=x \cdot \langle y, z\rangle_A$ for $z\in X$ is a generalized rank-one operator.
A representation $(\psi,\pi)$ is said to be \textit{Cuntz-Pimsner covariant} if
\begin{align}
\label{E:CP}
\psi^{(1)}(\phi(a))=\pi(a) \qforal a\in \phi^{-1}(\K(X)).
\end{align}

Recall that $X$ is said to be \textit{essential} if
$\ol\spn\{\phi(a) x: a\in A, x\in X\}=X$, and \textit{regular} if the left action $\phi$ is injective and $\phi(A)\subseteq \K(X)$.

\begin{defn}
Let $A$ be a C*-algebra, and let $X=\bigsqcup_{n \in \mathbb{N}^k}X_n$ be a semigroup such that $X_n$ is a C*-correspondence over $A$ for all $n \in \mathbb{N}^k$.
Then $X$ is called a \emph{product system over $\mathbb{N}^k$ with coefficient $A$} if
\begin{enumerate}
\item $X_0=A$;
\item $X_n \cdot X_m \subset X_{n+m}$ for all $n$, $m \in \mathbb{N}^k$;
\item for $n,m \in \mathbb{N}^k \setminus \{0\}$, there exists an isomorphism from $X_n \otimes_A X_m$ onto $X_{n+m}$,
where $X_n \otimes_A X_m$ denotes the balanced tensor product, by sending $x\otimes y$ to $xy$ for all $x \in X_n$ and $y \in X_m$;
\item for $n \in \mathbb{N}^k$, the multiplication $X_0 \cdot X_n$ is implemented by the left action of $A$ on $X_n$, and the multiplication $X_n \cdot X_0$ is implemented by the right action of $A$ on $X_n$.
\end{enumerate}
\end{defn}

\begin{defn}
\label{D:Toep}
Let $A$, $B$ be C*-algebras, $X$ be a product system over $\mathbb{N}^k$ with coefficient $A$, and let $\psi:X \to B$ be a map.
For $n \in \mathbb{N}^k$, denote by $\psi_n:=\psi \vert_{X_n}$. Then $\psi$ is called a (Toeplitz) \emph{representation} of $X$ if
\begin{enumerate}
\item[(T1)] $(\psi_{n},\psi_0)$ is a representation of $X_n$ for all $n \in \mathbb{N}^k$; and

\item[(T2)] $\psi_n(x)\psi_m(y)=\psi_{n+m}(xy)$ for all $n$, $m \in \mathbb{N}^k$, $x \in X_n$, $y \in X_m$.
\end{enumerate}
We write $\psi_n^{(1)}$ for the homomorphism from $\mathcal{K}(X_n)$ to $B$ as in \eqref{E:psi(1)}. The representation $\psi$ is said to be
\emph{Cuntz-Pimsner covariant} if $(\psi_{n},\psi_0)$ is Cuntz-Pimsner covariant for all $n \in \mathbb{N}^k$.

The product system $X$ is said to be \emph{essential} (resp.~\textit{regular}) if $X_n$ is essential (resp.~regular) for all $n \in \mathbb{N}^k$.
\end{defn}

\subsection*{Standing Assumptions:} \textsf{All product systems are always assumed to be essential and regular throughout the rest of the paper.}
Under these assumptions, every Cuntz-Pimsner covariant representation is automatically Nica covariant (see \cite[Proposition 5.4]{Fow02}).

\begin{prop}
Let $X$ be a product system over $\mathbb{N}^k$ with coefficient $A$.
Then there exists a universal Cuntz-Pimsner covariant representation $j_X:X \to \mathcal{O}_X$ such that $j_X$ generates $\mathcal{O}_X$, and for any  Cuntz-Pimsner
covariant representation $J$ of $X$ into a C*-algebra $B$, there is a unique homomorphism $\tilde J: \O_X\to B$ such that $\tilde J\circ j_X=J$.
The C*-algebra $\mathcal{O}_X$ is called the \emph{Cuntz-Pimsner algebra} of $X$.
\end{prop}

For a representation $\psi$ of $X$, a \emph{gauge action} is a strongly continuous homomorphism $\alpha:\prod_{i=1}^{k}\mathbb{T} \to \mathrm{Aut}(\ca(\psi(X)))$ such that $\alpha_z(\psi_n(x))=z^n \psi_n(x)$ for all $z \in \prod_{i=1}^{k}\mathbb{T}, n \in \mathbb{N}^k, x \in X_n$. The universal Cuntz-Pimsner covariant representation $j_X$ admits a gauge action $\gamma:\prod_{i=1}^{k}\mathbb{T} \to \mathrm{Aut}(\mathcal{O}_X)$. The gauge-invariant uniqueness theorem for a product system over $\mathbb{N}^k$ is highly nontrivial to achieve. However, this problem was completely resolved by Carlsen-Larsen-Sims-Vittadello in \cite{CLSV11} (their nice work covers much more general product systems). Combining \cite[Corollary~4.12]{CLSV11} and \cite[Corollary~5.2]{SY10}, we obtain the following version of the gauge-invariant uniqueness theorem, which is analogous to the one in \cite{FMR03}.

\begin{thm}\label{gauge-inv uni thm}
Let $X$ be a product system over $\mathbb{N}^k$ with coefficient $A$ and let $\psi$ be a Cuntz-Pimsner covariant representation of $X$ which admits a gauge action. Denote by
$h:\mathcal{O}_X \to \ca(\psi(X))$ the homomorphism induced from the universal property of $\mathcal{O}_X$. If $h \vert_{j_{X,0}(A)}$ is injective. Then $h$ is an isomorphism.
\end{thm}

For later use, let us record the following two simple lemmas.

 \begin{lem}\label{L:Toep}
 Let $A,B$ be C*-algebras where $A$ is generated by $\G$.
Let $X$ be a C*-correspondence over $A$,
which has a subset $\F$ whose linear span is dense in $X$.
Let $\psi_0:\spn\mathcal{F} \to B$ be a linear map, and $\pi:A \to B$ be a homomorphism. Suppose that
\begin{enumerate}
\item $\mathcal{G} \cdot \mathcal{F} \subset \mathcal{F}$;
\item $\psi_0(a \cdot x)=\pi(a)\psi_0(x)$ for all $x \in \mathcal{F}$ and $a \in \mathcal{G}$;
\item $\psi_0(x)^*\psi_0(y)=\pi(\langle x,y \rangle_A)$ for all $x,y \in \mathcal{F}$.
\end{enumerate}
Then $\psi_0$ is a bounded linear map with the unique extension $\psi$ to $X$, and $(\psi,\pi)$ is a representation of $X$. Moreover, if
$\psi^{(1)}(\phi(a))=\pi(a)$ for all $a \in \mathcal{G}$, then $(\psi,\pi)$ is also Cuntz-Pimsner covariant.
\end{lem}
\begin{proof}
It is straightforward to prove and left to the reader.
\end{proof}

\begin{lem}\label{L:CP}
Let $X$ be a product system over $\mathbb{N}^k$ with coefficient $A$ and let $\psi:X \to B$ be a representation.
Suppose that $(\psi_{e_i},\psi_0)$ is Cuntz-Pimsner covariant for all $1 \leq i \leq k$. Then $\psi$ is Cuntz-Pimsner covariant.
\end{lem}
\begin{proof}
For $1 \leq i,j \leq k$, there exists an isomorphism from $X_{e_i} \otimes X_{e_j}$ onto $X_{e_i+e_j}$ by sending $x \otimes y$ to $x \cdot y$ for all $x \in X_{e_i},y \in X_{e_j}$, and there exists a linear map $\psi_{e_i} \otimes \psi_{e_j}:X_{e_i} \otimes X_{e_j} \to B$ such that $\psi_{e_i} \otimes \psi_{e_j} (x \otimes y)=\psi_{e_i}(x)\psi_{e_j}(y)$ for all $x \in X_{e_i}$ and
$y \in X_{e_j}$. Similar to the proof of \cite[Lemma~3.10]{Pim97}, one can see that $(\psi_{e_i} \otimes \psi_{e_j},\psi_0)$ is Cuntz-Pimsner covariant.
Hence $\psi$ is Cuntz-Pimsner covariant.
\end{proof}

\subsection{Topological $k$-graphs}

In this subsection we recall the definition of topological $k$-graphs from  \cite{Yee07} which are generalizations of $k$-graphs studied by Kumjian-Pask in \cite{KP00}. Then we briefly
recall  the product system associated to each topological $k$-graph from \cite{Yam09}.

\begin{defn}
A \emph{topological $k$-graph} is a $k$-graph $\Lambda$ equipped with a locally compact Hausdorff topology such that
\begin{itemize}
\item the composition of paths is continuous and open,
\item the range map $r$, the source map $s$, and the degree map $d$ are all continuous,
\item $s$ is a local homeomorphism.
\end{itemize}
For $n \in \mathbb{N}^k$, let $\Lambda^n$ be the set of all paths of degree $n$.
The topological $k$-graph $\Lambda$ is said to be \emph{regular} if $r \vert_{\Lambda^{e_i}}$ is proper and surjective for all $1 \leq i \leq k$.
\end{defn}

One can show that if $\Lambda$ is regular then $r \vert_{\Lambda^n}$ is proper and surjective for all $n \in \mathbb{N}^k$.

Let $\Lambda$ be a regular topological $k$-graph. One can construct a product system $X(\Lambda)$ over $\bN^k$ as follows:
Given $n\in \bN^k$, define a topological graph $E_n:=(\Lambda^0,\Lambda^n,r,s)$.
Let $X_n(\Lambda):=X(E_n)$ be the graph correspondence of $E_n$ in the sense of Katsura (cf. \cite{Katsura:TAMS04}). By \cite[Proposition~1.10]{Katsura:TAMS04},
\[
X_n(\Lambda)=\{x \in \rC(\Lambda^n):\langle x,x \rangle_{C_0(\Lambda^0)} \in C_0(\Lambda^0)\}.
\]
For $n,m \in \mathbb{N}^k, x \in X_n$ and $y \in X_m$, define a diamond operation $x \diamond y:\Lambda^{n+m} \to \mathbb{C}$ by
\[
x \diamond y(\mu):=x(\alpha)y(\beta)\text{ for } \mu \in \Lambda^{n+m} \text{ with }\mu=\alpha\beta,d(\alpha)=n,d(\beta)=m.
\]
Notice that $x\diamond y$ is well-defined due to the unique factorization of $\mu$.
Let
 \[
 X(\Lambda):=\bigsqcup_{n \in \mathbb{N}^k}X_n(\Lambda).
 \]
 Then $X(\Lambda)$ is a product system over $\mathbb{N}^k$ with coefficient $C_0(\Lambda^0)$ under $\diamond$.
 We call $X(\Lambda)$ the \textit{product system associated to $\Lambda$}. Notice that $X(\Lambda)$ is essential and regular.


\subsection{Single-vertex $k$-graphs}
\label{SS:single}

In this subsection we recap the theory of single-vertex $k$-graphs and their C*-algebras from \cite{DY091, DY09}. Let $\Lambda$ be a single-vertex $k$-graph. For $1\le i\le k$, let
$
\{x^i_\fs:\fs\in[m_i]\}
$
be the set of all edges in $\Lambda$ of degree $e_i$.
It follows from the factorization property of $\Lambda$ that,
for $1\le i<j\le k$, there is a bijection $\theta_{ij}: [m_i]\times[m_j]\to [m_j]\times [m_i]$
satisfying the following \textit{$\theta$-commutation relations}
\begin{align*}
 x^i_\fs x^j_\ft = x^j_{\ft'} x^i_{\fs'}
 \quad\text{if}\quad
 \theta_{ij}(\fs,\ft) = (\ft',\fs').
\end{align*}
Then $\Lambda$ coincides with the semigroup $\Fth$ defined as follows
\cite{DY09}
\begin{align*}
\Fth=\big\langle x_\fs^i: \fs\in [m_i],\, 1\le i\le k;\,  x^i_\fs x^j_\ft = x^j_{\ft'} x^i_{\fs'}
 \text{ whenever }
 \theta_{ij}(\fs,\ft) = (\ft',\fs') \big\rangle,
\end{align*}
which is also occasionally written as
\begin{align*}
\Fth=\big\langle x_\fs^i:  \fs\in [m_i],\ 1\le i\le k; \, \theta_{ij}, \, 1\le i<j\le k \big\rangle.
\end{align*}
It is worthwhile to mention that $\Fth$ has the cancellation property due to the factorization property of $\Lambda$.
It follows from the $\theta$-commutation relations that every element $w\in \Fth$ has the normal form
$
w=x_{u_1}^1\cdots x_{u_k}^k
$
for some $u_i\in\bF_{m_i}^+$ ($1\le i\le k$). Here we use
the multi-index notation: $x^i_{u_i}=x^i_{\fs_1}\cdots x^i_{\fs_n}$ if $u_i=\fs_1\cdots\fs_n\in\bF_{m_i}^+$.

For $k=2$, every permutation $\theta$ determines a 2-graph. But for $k\ge 3$, $\theta=\{\theta_{ij}:1\le i<j\le k\}$ determines a
$k$-graph if and only if it satisfies a \textit{cubic condition} (see, e.g., \cite{DY09, FS02} for its definition).
Here it is probably worth mentioning that this is also related to the Yang-Baxter equation (see \cite{Yan161, Yan162}).

By a \textit{$*$-representation} $S$ of $\Fth$ in a C*-algebra $\A$, we mean that $S$ is a semigroup homomorphism of $\Fth$ which subjects to the relations:
for $1\le i\le k$, $S_{x^i_\fs}^*S_{x^i_\fs}=1$ $(\fs\in[m_i]$), and the defect free condition
$
\sum_{\fs\in[m_i]}S_{x^i_\fs}S_{x^i_\fs}^*=1.
$
The \textit{$k$-graph C*-algebra} $\O_\theta$ of $\Fth$ is defined to be the universal C*-algebra for $*$-representations of $\Fth$.

\subsection{Zappa-Sz\'{e}p products of semigroups}

In this subsection we review the definitions of the full C*-algebra of a left cancellative semigroup from \cite{Li12}, its boundary quotient C*-algebra from \cite{BRRW14}, and
 the Zappa-Sz\'{e}p product of two semigroups from \cite{BRRW14} (see also \cite{Bri05}). The odometer action is also given to induce a class of Zappa-Sz\'{e}p products.

Let $P$ be a left cancellative semigroup.
For $p \in P$, we also denote by $p$ the left multiplication map $q \mapsto pq$. The set of \emph{constructible right ideals} is defined as
\[
\mathcal{J}(P):=\{p_1^{-1}q_1\cdots p_n^{-1}q_nP:n \geq 1, p_1,q_1,\dots, p_n,q_n \in P\} \cup \{\mt\}.
\]

A finite subset $F$ of $\mathcal{J}(P)$ is called a \emph{foundation set} if for each $Y \in \mathcal{J}(P)$ there exists $X \in F$ such that $X \cap Y \neq \mt$.

For $p,q \in P$, we say that $p$ is a \textit{right multiple} of $q$ if there exists $r \in P$ such that $p=qr$. $P$ is said to be \emph{right LCM} if any two elements of $P$ having a right common multiple have a right least common multiple.

\begin{defn}[{\cite[Definition 2.2]{Li12}, \cite[Definition 5.1]{BRRW14}}]
\label{D:Li}
Given a left cancellative semigroup $P$, the \emph{full semigroup C*-algebra} $\ca(P)$ of $P$ is the universal C*-algebra generated by a family of isometries $\{v_p\}_{p \in P}$ and a family of projections $\{e_X\}_{X \in \mathcal{J}(P)}$ satisfying the following relations:
\begin{itemize}
\item[(L1)] $v_p v_q=v_{pq}$ for all $p,q \in P$;
\item[(L2)]  $v_p e_X v_p^*=e_{pX}$ for all $p \in P,X \in \mathcal{J}(P)$;
\item[(L3)]  $e_\mt=0$ and $e_P=1$;
\item[(L4)]  $e_X e_Y=e_{X \cap Y}$ for all $X,Y \in \mathcal{J}(P)$.
\end{itemize}

The \textit{boundary quotient} $\Q(P)$ of $\ca(P)$ is the universal C*-algebra generated by a family of isometries $\{v_p\}_{p \in P}$ and a family of projections $\{e_X\}_{X \in \mathcal{J}(P)}$ satisfying Conditions (L1)--(L4), and furthermore
\begin{itemize}
\item[(Q5)]
$\displaystyle \prod_{X \in F}(1-e_X)=0$ for all foundation sets $F \subset \mathcal{J}(P)$.
\end{itemize}
\end{defn}

In this paper, $\Q(P)$ is simply called the \textit{boundary quotient C*-algebra of $P$}.

\begin{defn}
[{\cite[Definition~3.1]{BRRW14}}]
\label{D:Z-S}
Let $U$ and $A$ be semigroups. Suppose there are two maps $A \times U \to U, (a,u)\mapsto a\cdot u$ and $A \times U \to A, (a,u)\mapsto a|_u$ such that
for all $a$, $b \in A$ and $u$, $v \in U$, we have
\begin{align*}
&(\text{B1})\ 1_A \cdot u=u; &  & (\text{B5})\ a \vert_{1_U}=a;\\
&(\text{B2})\ (ab) \cdot u=a \cdot(b \cdot u); & & (\text{B6})\ a \vert_{uv}=a \vert_u \vert_v;\\
&(\text{B3})\ a \cdot 1_U=1_U;  & & (\text{B7})\ 1_A \vert_u=1_A;\\
&(\text{B4})\ a \cdot (uv)=(a \cdot u)(a \vert_u \cdot v);  & & (\text{B8})\ (ab) \vert_u=a \vert_{b \cdot u} b\vert_u.
\end{align*}
Denote by $U \bowtie A:=U \times A$ equipped with the multiplication
\[
(u,a)(v,b):=(u (a \cdot v),a|_v\cdot b) \qforal (u,a),(v,b) \in U \times A.
\]
Then $U \bowtie A$ is a semigroup under this multiplication, which is called the \emph{(external) Zappa-Sz\'{e}p product}.

We call $a\cdot u$ the \textit{action} of $a$ on $u$, and $a|_u$ the \textit{restriction} of $a$ to $u$.
\end{defn}

Let us record the following remark for later use.

\begin{rem}
\label{R:lc}
If $U$ and $A$ in Definition \ref{D:Z-S} are both left cancellative semigroups, and if for any $a \in A$, the map $u \mapsto a \cdot u$ is an injection on $U$, then
$U \bowtie A$ is also left cancellative.
\end{rem}

One very useful way to produce Zappa-Sz\'ep products is from self-similar actions.

\begin{defn}
[{\cite[Definition~1.5.1]{Nek05}}]
Let $X$ be a non-empty finite set. Consider the free semigroup $X^*$ generated by $X$.
Suppose that a group $G$ acts faithfully on $X^*$. Then this action is called \emph{self-similar} if
\begin{enumerate}
\item $g \cdot \mt=\mt$ for all $g \in G$;
\item for $g \in \mathbb{Z}, x \in X$, there exist unique $y \in X, h \in G$ such that $g \cdot (xw)=y(h \cdot w)$ for all $w \in X^*$.
\end{enumerate}
We also call $(G, X)$ a \textit{self-similar action}.
\end{defn}

If we let $g\cdot x:=y$ and $g|_x:=h$, then these two maps induce two maps
$G\times X^*\to X^*, (g,u)\mapsto g \cdot u$ and $G\times X^*\to G, (g,u)\mapsto g \vert_u$
satisfying Conditions (B1)--(B8) of Definition \ref{D:Z-S}.
Identifying $X^*$ with $\bF_{|X|}^+$, we obtain a Zappa-Sz\'ep product $\bF_{|X|}^+\bowtie G$.

A very important example of self-similar actions (see \cite{LRRW14, Nek05}), which will be frequently used later, is given below.

\begin{eg}[\textbf{Odometers}]
\label{Eg:odometer}
Let $n\ge 1$ and $X=\{x_\fs:\fs\in[n]\}$. Define
\begin{align*}
1\cdot x_\fs&=x_{(\fs+1)\text{ mod\,} n}\quad \text{for }\fs\in[n],\\
1|_{x_\fs}&=
\begin{cases}
0\quad \text{if }\fs<n-1\\
1\quad \text{if }\fs=n-1.
\end{cases}
\end{align*}
This determines a self-similar action $(\bZ, X)$, which is known as an \textit{odometer} or an \textit{adding machine}.
\end{eg}

\section{Generators and relations of $\mathcal{Q}(\mathbb{F}_\theta^+ \bowtie G)$}
\label{S:genrel}

When applying the construction given in Definition \ref{D:Li} to the Zappa-Sz\'ep product $U\bowtie A$ of two semigroups $A$ and $U$,
usually we find it hard to understand its boundary quotient C*-algebra $\Q(U\bowtie A)$. This is not surprising due to several factors:
for instance, the constructible right ideals of $U\bowtie A$ could be very complex; its foundation sets are not easy to describe.

In this section, we study a class of Zappa-Sz\'ep products $\Fth\bowtie G$, where $G$ is a group and $\Fth$ is a single-vertex $k$-graph such that the restriction map satisfies a certain condition. In this case, $\Q(\Fth\bowtie G)$ can be nicely presented by a unitary representation of $G$ and a $*$-representation
of $\Fth$ such that they are compatible with the action and restriction maps.

\begin{lem}
\label{L:Q(UxA)}
Let $U$ be a left cancellative semigroup and $G$ be a group.
Let $G\times U\to U, (a,u)\mapsto a\cdot u$ and $G\times U\to G, (a,u)\mapsto a|_u$ be two maps
satisfying Conditions (B1)--(B8) of Definition \ref{D:Z-S}.
Suppose that for $u \in U, b \in G$, there exists $a \in G$ such that $a \vert_u=b$.
Then $\mathcal{Q}(U \bowtie G)$ is isomorphic to the universal C*-algebra $\fA$ generated by a family of isometries $\{t_u \}_{u \in U}$, a family of projections $\{q_X\}_{X \in \mathcal{J}(U)}$, and a family of unitaries $\{s_a\}_{a \in G}$ satisfying the following properties:
For $u,v \in U, X,Y \in \mathcal{J}(U), a,b \in G$,
\begin{enumerate}
\item\label{t_u t_v=t_uv} $t_u t_v=t_{uv}$;
\item\label{t_u q_X t_u*=q_{uX}} $t_u q_X t_u^*=q_{uX}$;
\item $s_a q_X s_a^*=q_{a \cdot X}$;\footnote{It will be shown that $a \cdot X \in \mathcal{J}(U)$ later.}
\item $q_\mt=0$ and $q_U=1$;
\item\label{q_X q_Y=q_X cap Y} $q_X q_Y=q_{X \cap Y}$;
\item\label{prod(1-q_X)=0} $\prod_{X \in F}(1-q_X)=0$ for every foundation set $F \subset \mathcal{J}(U)$;
\item $s_a s_b=s_{ab}$;
\item\label{s_a t_u=t_a cdot u s_a vert u} $s_a t_u=t_{a \cdot u} s_{a \vert_u}$.
\end{enumerate}
\end{lem}

\begin{proof}
First of all, we investigate the constructible right ideals $\mathcal{J}(U \bowtie G)$ of $U \bowtie G$. For $u,v \in U, a,b \in G, S \subset U$, one can easily verify that
\begin{align}
\nonumber
(u,a)(S \times G)&=(u(a \cdot S)) \times G;\\
\label{E:b cdot v^-1 uS}
b \cdot (v^{-1}uS)&=(c \cdot v)^{-1}(c \cdot (uS)) \text{ for } c \in A \text{ with } c|_v=b;\\
\nonumber
(v,b)^{-1}(S \times G)&=(b^{-1} \cdot (v^{-1}S)) \times G.
\end{align}
One can immediately see the following facts from \eqref{E:b cdot v^-1 uS}:
\begin{itemize}
\item $a\cdot X\in \J(U)$ for any $a\in G$ and $X\in \J(U)$;
\item $\mathcal{J}(U \bowtie G)=\{X \times G:X \in \mathcal{J}(U)\}$; and
\item for any finite subset $F \subset \mathcal{J}(U)$, $F$ is a foundation set if and only if $\{X \times G\}_{X \in F}$ is a foundation set of $\mathcal{J}(U \bowtie G)$.
\end{itemize}

Let $\{\delta_{(u,a)}\}_{(u,a) \in U \bowtie G}$ (resp.~$\{e_{X \times G}\}_{X \in \mathcal{J}(U)}$) be the families of isometries (resp.~projections) which generate $\mathcal{Q}(U \bowtie G)$.

For $(u,a) \in U \bowtie G$ and $X \in \mathcal{J}(U)$, define
\[
\Delta_{(u,a)}:=t_u s_a \text{ and } E_{X \times G}:=q_X.
\]

Given $(u,a),(v,b) \in U \bowtie G$ and $X, Y \in \mathcal{J}(U)$, we have the following properties:
\begin{itemize}
\item $\Delta_{(u,a)} \Delta_{(v,b)}=t_u s_a t_v s_b=t_{u (a \cdot v)}s_{a \vert_v b}=\Delta_{(u,a)(v,b)}$.
\item $\Delta_{(u,a)}E_{X \times G}\Delta_{(u,a)}^*=t_u s_a q_X s_a^* t_u^*=q_{u(a \cdot X)}=E_{u a \cdot X \times G}=E_{(u,a)X \times G}$.
\item $E_\mt=q_\mt=0$ and $E_{U \bowtie G}=q_U=1$.
\item $E_{X \times G} E_{Y \times G}=q_X q_Y=q_{X \cap Y}=E_{(X \times G)\cap(Y \times G)}$.
\item For a foundation set $\{X_i \times G\}_{1 \leq i \leq n}$ of $\mathcal{J}(U \bowtie G)$, since $\{X_i\}_{i=1}^{n}$ is a foundation set of $\mathcal{J}(U)$, we have
\[
\prod_{i=1}^{n}(1-E_{X_i \times G})=\prod_{i=1}^{n}(1-q_{X_i})=0.
\]
\end{itemize}
Hence Relations (L1)--(L4) and (Q5) of Definition \ref{D:Z-S} hold. By the universal property of $\mathcal{Q}(U \bowtie G)$, there exists a homomorphism $\rho:\mathcal{Q}(U \bowtie G) \to \fA$ such that $\rho(\delta_{u,a})=\Delta_{(u,a)}$ and $\rho(e_{X \times G})=E_{X \times G}$ for all $(u,a) \in U \bowtie G$ and $X \in \mathcal{J}(U)$.

Conversely, for $u \in U, X \in \mathcal{J}(U)$ and $a \in G$, define
\[
T_u:=\delta_{(u,1_G)}, \ Q_X:=e_{X \times G}, \ S_a:=\delta_{(1_U,a)}.
\]
For $a \in G$, we compute that
\[
S_a S_a^*:=\delta_{(1_U,a)} e_{U \bowtie G} \delta_{(1_U,a)}^*=e_{(1_U,a)(U \bowtie G)}=e_{U \bowtie G}=1.
\]
So $S_a$ is unitary.

Given $u,v \in U, X,Y \in \mathcal{J}(U)$ and $a,b \in G$, the following properties hold:
\begin{itemize}
\item $T_u T_v=\delta_{(u,1_G)}\delta_{(v,1_G)}=\delta_{(uv,1_G)}=T_{uv}$.
\item $T_u Q_X T_u^*=\delta_{(u,1_G)}e_{X \times G} \delta_{(u,1_G)}^*=e_{(u,1_G)(X \times G)}=e_{uX \times G}=Q_{uX}$.
\item $S_a Q_X S_a^*=\delta_{(1_U,a)} e_{X \times G} \delta_{(1_U,a)}^*=e_{(1_U,a)(X \times G)}=e_{(a \cdot X) \times G}=Q_{a \cdot X}$.
\item $Q_\mt=e_\mt=0$ and $Q_U=e_{U \times G}=1$.
\item $Q_X Q_Y=e_{X \times G}e_{Y \times G}=e_{(X \cap Y) \times G}=Q_{X \cap Y}$.
\item For any foundation set $F$ of $\mathcal{J}(U)$, since $\{X \times G\}_{X \in F}$ is a foundation set of $\mathcal{J}(U \bowtie G)$, we have
$
\prod_{X \in F}(1-Q_X)=\prod_{X \in F}(1-e_{X \times G})=0.
$
\item $S_a S_b=\delta_{(1_U,a)}\delta_{(1_U,b)}=\delta_{(1_U,ab)}=S_{ab}$.
\item $S_a T_u=\delta_{(1_U,a)}\delta_{(u,1_G)}=\delta_{(a \cdot u,1_G)}\delta_{(1_U,a \vert_u)}=T_{a \cdot u} S_{a \vert_u}$.
\end{itemize}
So Conditions~(\ref{t_u t_v=t_uv})--(\ref{s_a t_u=t_a cdot u s_a vert u}) hold. By the universal property of $\fA$, there exists a homomorphism $\pi:\fA \to \mathcal{Q}(U \bowtie G)$ such that $\pi(t_u)=T_u,\pi(q_X)=Q_X, \pi(s_a)=S_a$ for all $u \in U, X\in \mathcal{J}(U), a\in G$.

Finally, it is straightforward to see that $\pi \circ \rho=\id$ and $\rho \circ \pi=\id$. Therefore $\mathcal{Q}(U \bowtie G)$ is isomorphic to $\fA$.
\end{proof}

\begin{lem}\label{constructible right ideals of F_theta}
Let $\Fth$ be a single-vertex $k$-graph, and $T$ be a $*$-representation of $\Fth$ in a C*-algebra $\mathcal{A}$. Given $\mu_1,\nu_1,\dots,\mu_n,\nu_n$ in  $\Fth$, denote by
\[
F:=\left\{(\alpha_i,\beta_i)_{i=1}^{n} \in \prod_{j=1}^{2n}\Fth:(\alpha_i,\beta_i) \in \begin{cases}
     (\Fth)^{\min}(\mu_i,\nu_i \alpha_{i-1}) &\text{if $1 < i \leq n$} \\
    (\Fth)^{\min}(\mu_i,\nu_i) &\text{if $i=1$}
\end{cases}\right\},
\]
where $(\Fth)^{\min}(\mu,\nu)$ denotes the set of minimal common extensions of $\mu$ and $\nu$. Then the following statement hold true. 
\begin{enumerate}
\item\label{alpha_n neq gamma_n beta_n neq omega_n} For distinct tuples $(\alpha_i,\beta_i)_{i=1}^{n}$ and $(\gamma_{i},\omega_{i})_{i=1}^{n} \in F$, we have $d(\alpha_n)=d(\gamma_n)$, $d(\beta_n)=d(\omega_n)$, $\alpha_n \neq \gamma_n$, and $\beta_n \neq \omega_n$.
\item\label{ideal structure of Lambda} $\mu_n^{-1}\nu_n \cdots \mu_1^{-1} \nu_1 \Fth=\bigcup_{(\alpha_i,\beta_i)_{i=1}^{n} \in F} \alpha_n \Fth$.
\item\label{unique presentation of union} Each constructible right ideal of $\bF_\theta^+$ has a unique representation as the union of disjoint principal right ideals of $\bF_\theta^+$.
\item\label{J(F_theta^+)} $\mathcal{J}(\bF_\theta^+)=\Big\{\bigcup_{i=1}^{n}\alpha_i \bF_\theta^+:d(\alpha_1)=\cdots =d(\alpha_n)\Big\}$.
\item\label{foundation=exhaustive} For any finite subset $F \subset \bF_\theta^+$, we have $\{\alpha \bF_\theta^+\}_{\alpha \in F}$ is a foundation set of $\mathcal{J}(\bF_\theta^+)$ if and only if $F$ is exhaustive (see \cite[Definition~2.4]{RSY04}).
\end{enumerate}
\end{lem}

\begin{proof}
(\ref{alpha_n neq gamma_n beta_n neq omega_n}) follows from the unique factorization property of $\Fth$.

We verify (\ref{ideal structure of Lambda}) by induction. It is straightforward to see that (\ref{ideal structure of Lambda}) holds for $n=1$. Suppose that (\ref{ideal structure of Lambda}) holds for $n \geq 1$.
Let
\[
F':=\left\{(\alpha_i,\beta_i)_{i=1}^{n+1} \in \prod_{j=1}^{2(n+1)}\Fth:(\alpha_i,\beta_i) \in \begin{cases}
     (\Fth)^{\min}(\mu_i,\nu_i \alpha_{i-1}) &\text{if $i>1$} \\
     (\Fth)^{\min}(\mu_i,\nu_i) &\text{if $i=1$}
\end{cases}\right\}.
\]
Then
\begin{align*}
&\mu_{n+1}^{-1}\nu_{n+1}\mu_n^{-1}\nu_n \cdots \mu_1^{-1} \nu_1 \Fth
\\&=\mu_{n+1}^{-1}\nu_{n+1}\Big(\bigcup_{(\alpha_i,\beta_i)_{i=1}^{n} \in F} \alpha_n \Fth \Big)
\\&=\bigcup_{(\alpha_i,\beta_i)_{i=1}^{n} \in F}\mu_{n+1}^{-1}\nu_{n+1}(\alpha_n \Fth)
\\&=\bigcup_{(\alpha_i,\beta_i)_{i=1}^{n} \in F}\mu_{n+1}^{-1}\nu_{n+1}\alpha_n \Fth
\\&=\bigcup_{(\alpha_i,\beta_i)_{i=1}^{n} \in F}\bigcup_{(\alpha_{n+1},\beta_{n+1}) \in (\Fth)^{\min}(\mu_{n+1},\nu_{n+1}\alpha_{n})}\alpha_{n+1}\Fth
\\&=\bigcup_{(\alpha_i,\beta_i)_{i=1}^{n+1} \in F'}\alpha_{n+1}\Fth.
\end{align*}
So this proves (\ref{ideal structure of Lambda}). (\ref{unique presentation of union})--\eqref{foundation=exhaustive} follow from (\ref{ideal structure of Lambda}) easily.
\end{proof}



\begin{thm}
\label{T:kQ(LambdaxG)}
Let $\Fth$ be a single-vertex $k$-graph,  $G$ be a group, and
let $G\times \Fth\to \Fth, (g,\mu)\mapsto g\cdot \mu$ and $G\times \Fth\to G, (g,\mu)\mapsto g|_\mu$ be two maps
satisfying Conditions (B1)--(B8) of Definition
\ref{D:Z-S}. Suppose that for $\mu \in \Fth, h \in G$, there exists $g \in G$ such that $g \vert_\mu=h$. Then $\mathcal{Q}(\Fth \bowtie G)$ is isomorphic to the universal C*-algebra $\mathcal{A}$ generated by a unitary representation $u$ of $G$ and a $*$-representation $v$ of $\Fth$ satisfying
\begin{align}
\label{E:ST}
u_g v_\mu=v_{g \cdot \mu} u_{g \vert_\mu}\qforal \mu \in \Fth \text{ and } g\in G.
\end{align}
\end{thm}

\begin{proof}
We apply the characterization of $\mathcal{Q}(\Fth \bowtie G)$ from Lemma~\ref{L:Q(UxA)}. That is, $\mathcal{Q}(\Fth \bowtie G)$ is the universal C*-algebra generated by a family of isometries $\{t_\mu\}_{\mu \in \bF_\theta^+}$, a family of projections $\{q_X\}_{X \in \mathcal{J}(\bF_\theta^+)}$, and a family of unitaries $\{s_a\}_{a \in G}$ satisfying Conditions~(\ref{t_u t_v=t_uv})--(\ref{s_a t_u=t_a cdot u s_a vert u}) of Lemma~\ref{L:Q(UxA)}.

First of all, for $\mu \in \bF_\theta^+$, $g\in G$ and $X=\bigcup_{i=1}^{n}\alpha_i \bF_\theta^+ \in \mathcal{J}(\bF_\theta^+)$, define 
$
T_\mu:=v_\mu,  S_g:=u_g, Q_X:=\sum_{i=1}^{n}v_{\alpha_i} v_{\alpha_i}^*.
$
It is clear that $T_\mu$ and $S_g$ are isometric and unitary respectively. Also notice that $Q_X$ is a well-defined projection due to Lemma~\ref{constructible right ideals of F_theta}. In what follows, we only verify that $\{T_\mu,S_g,Q_X:\mu\in\Fth, g\in G, X\in\J(\Fth)\}$ satisfies Conditions~(\ref{q_X q_Y=q_X cap Y}) and (\ref{prod(1-q_X)=0}) of Lemma~\ref{L:Q(UxA)}, as the other conditions hold easily.

To prove Condition~(\ref{q_X q_Y=q_X cap Y}) of Lemma~\ref{L:Q(UxA)}, let us fix $X=\bigcup_{i=1}^{n}\alpha_i \bF_\theta^+$ and $Y=\bigcup_{j=1}^{m}\beta_j \bF_\theta^+ $ in 
$\mathcal{J}(\bF_\theta^+)$. Then $X \cap Y=\bigcup_{i,j}\bigcup_{(\mu,\nu) \in \Lambda^{\min}(\alpha_i,\beta_j)}\alpha_i \mu \bF_\theta^+$. So
\begin{align*}
Q_X Q_Y&=\sum_{i,j}v_{\alpha_i} v_{\alpha_i}^*v_{\beta_j} v_{\beta_j}^*
\\&=\sum_{i,j}\sum_{(\mu,\nu) \in (\Fth)^{\min}(\alpha_i,\beta_j)}v_{\alpha_i \mu}v_{\beta_j \nu}^* \text{ (see \cite[Lemma~3.1]{KP00})}
\\&=Q_{X \cap Y}.
\end{align*}

For the proof of Condition~(\ref{prod(1-q_X)=0}) of Lemma~\ref{L:Q(UxA)}, pick a foundation set $\{X_i:=\bigcup_{j=1}^{m_i}\alpha_{ij} \bF_\theta^+\}_{i=1}^{n}$ of $\mathcal{J}(\bF_\theta^+)$. Notice that $\{\alpha_{ij} \bF_\theta^+:1 \leq i \leq n, 1 \leq j \leq m_i\}$ is also a foundation set of $\mathcal{J}(\bF_\theta^+)$. By Lemma~\ref{constructible right ideals of F_theta}, $\{\alpha_{ij}:1 \leq i \leq n, 1 \leq j \leq m_i\}$ is exhaustive. Then it follows from \cite[Proposition B.1]{RSY04}) that
\begin{align*}
\prod_{i=1}^{n}(1-Q_{X_i})&=\prod_{i=1}^{n} \prod_{j=1}^{m_i}(1-v_{\alpha_{ij}} v_{\alpha_{ij}}^*)=0 . 
\end{align*}

By the universal property of $\mathcal{Q}(\Fth \bowtie G)$, there exists a homomorphism $\pi:\mathcal{Q}(\Fth \bowtie G) \to \mathcal{A}$ such that $\pi(t_\mu)=T_{\mu}, \pi(s_g)=S_g,\pi(q_X)=Q_X$ for all $\mu \in \bF_\theta^+, g \in G, X \in \mathcal{J}(\bF_\theta^+)$.

Conversely, let 
\[
V_\mu:=t_\mu, \ U_g:=s_g \qforal \mu\in\Fth, \ g\in G.
\]
Clearly, $V_\mu$ is an isometry and $U_g$ is a unitary. 
We verify that $V$ is a $*$-representation of $\Fth$. 
Obviously, we only need to show that $\sum_{\mu \in (\bF_\theta^+)^{e_i}}V_\mu V_\mu^*=1$ for all $1 \leq i \leq k$. To this end, let $1 \leq i \leq k$. For distinct $\mu,\nu \in (\bF_\theta^+)^{e_i}$, by Condition~(\ref{t_u q_X t_u*=q_{uX}}) of Lemma~\ref{L:Q(UxA)}, we get $t_\mu t_\mu^*=q_{\mu \bF_\theta^+}$ and $t_\nu t_\nu^*=q_{\nu \bF_\theta^+}$. By Condition~(\ref{q_X q_Y=q_X cap Y}) of Lemma~\ref{L:Q(UxA)}, we have $t_\mu t_\mu^*t_\nu t_\nu^*=0$. Since $\{\mu F_\theta^+\}_{\mu \in (\bF_\theta^+)^{e_i}}$ is a foundation set of $\mathcal{J}(\bF_\theta^+)$, we have
\begin{align*}
1-\sum_{\mu \in (\bF_\theta^+)^{e_i}}V_\mu V_\mu^*&=1- \sum_{\mu \in (\bF_\theta^+)^{e_i}}t_\mu t_\mu^*
\\&=\prod_{\mu \in (\bF_\theta^+)^{e_i}}(1-t_\mu t_\mu^*)
\\&=\prod_{\mu \in (\bF_\theta^+)^{e_i}}(1-q_{\mu \bF_\theta^+}) \text{ (by Lemma~\ref{L:Q(UxA)} (\ref{prod(1-q_X)=0}))}
\\&=0.
\end{align*}
Thus, by the universal property of $\mathcal{A}$, there exists a homomorphism $\rho: \mathcal{A} \to \mathcal{Q}(\Fth \bowtie G)$ such that $\rho(v_\mu)=V_\mu$, $\rho(u_g)=U_g$
for all $\mu \in F_\theta^+, g \in G$. 

It remains to show that $\pi$ and $\rho$ are inverse of each other. For this, let $X:=\bigcup_{i=1}^{n}\alpha_i \bF_\theta^+ \in \mathcal{J}(\bF_\theta^+)$. 
Denote by $F:=(\bF_\theta^+)^{d(\alpha_1)} \setminus \{\alpha_i\}_{i=1}^{n}$. Then 
$\{\alpha_i \bF_\theta^+,\alpha \bF_\theta^+:1 \leq i \leq n, \alpha \in F\}$ and 
$\{X,\alpha \bF_\theta^+:\alpha \in F\}$ 
are foundation sets of $\mathcal{J}(\bF_\theta^+)$. By Conditions~(\ref{q_X q_Y=q_X cap Y})--(\ref{prod(1-q_X)=0}) of Lemma~\ref{L:Q(UxA)}, we have
\begin{align*}
\prod_{i=1}^{n}\big(1-q_{\alpha_i \bF_\theta^+}\big)\prod_{\alpha \in F}\big(1-q_{\alpha \bF_\theta^+}\big)
&=1-\sum_{i=1}^{n}q_{\alpha_i \bF_\theta^+}-\sum_{\alpha \in F}q_{\alpha \bF_\theta^+}=0,\\
\big(1-q_X\big)\prod_{\alpha \in F}\big(1-q_{\alpha \bF_\theta^+}\big)
&=1-q_X-\sum_{\alpha \in F}q_{\alpha \bF_\theta^+}=0.
\end{align*}
So $\rho \circ \pi (q_X)=q_X$. Then it is easy to see that $\rho \circ \pi=\id, \pi \circ \rho=\id$. Therefore we are done.
\end{proof}

Two remarks are recorded below.

\begin{rem}
Theorem \ref{T:kQ(LambdaxG)} is an analogue  of \cite[Theorem 5.2]{BRRW14}. However, since a single-vertex $k$-graph $\Fth$
is not necessarily right LCM in general (also see Proposition \ref{P:iffLCM} below), the assumptions of \cite[Theorem~5.2]{BRRW14} are not satisfied in our case.
So here one can not apply \cite[Theorem~5.2]{BRRW14}.
\end{rem}

\begin{rem}
If $G$ is trivial, Theorem \ref{T:kQ(LambdaxG)} implies that the boundary quotient C*-algebra $\mathcal{Q}(\Fth)$ is isomorphic to the graph C*-algebra $\O_\theta$ of $\Fth$.
We should also mention that the C*-algebra $\ca(\Fth)$ in \cite{DSY08, DSY10}  is really $\Q(\Fth)$ here,
instead of the full C*-algebra of (the semigroup) $\Fth$.
(To avoid confusion, the notation $\O_\theta$ was first used in \cite{Yan10}.)
\end{rem}

\section{An application to the standard products of Odometers}
\label{S:pssa}

Applying the main result in Section \ref{S:genrel} to the standard product of $k$ odometers, we further simplify the presentation of
its boundary quotient C*-algebra. Our result, loosely speaking, says that the boundary quotient C*-algebra in this case
is generated by a unitary representation of a group and a $*$-representation of a single-vertex $k$-graph which are compatible with
the odometer actions.

For our purpose, we first generalize \cite[Proposition 3.10]{BRRW14} to higher dimensional cases.

\begin{prop}[\textbf{and Definition}]
\label{P:kZ-S}
Let $G$ be a group and
\[
\Fth=\big\langle x_\fs^i: \fs\in [n_i],\, 1\le i\le k;\,  x^i_\fs x^j_\ft = x^j_{\ft'} x^i_{\fs'}
 \text{ whenever }
 \theta_{ij}(\fs,\ft) = (\ft',\fs') \big\rangle
\]
be a single-vertex $k$-graph.
Suppose that $G$ acts self-similarly on each $\mathbb{F}_{n_i}^+$ $(1\le i\le k)$.
Then the action and restriction maps $G\times X_i\to X_i, (g,x_\fs^i)\mapsto g\cdot x^i_\fs$ and $G\times X_i\to G, (g,x_\fs^i)\mapsto g|_{x^i_\fs}$
can be extended to $G\times \Fth\to \Fth, (g,\mu)\mapsto g\cdot \mu$ and $G\times \Fth\to G, (g,\mu)\mapsto g|_{\mu}$
satisfying Condition (B1)--(B8) in Definition \ref{D:Z-S}, if and only if
\begin{align}
\label{E:welldef}
(g \cdot x_\fs^i) (g \vert_{x_\fs^i} \cdot x_\ft^j)&=(g \cdot x_{\ft'}^j)(g \vert_{x_{\ft'}^j} \cdot x_{\fs'}^i)
\end{align}
for all generators $g$ of $G$ and $\theta_{ij}( x^i_\fs, x^j_\ft )= (x^j_{\ft'}, x^i_{\fs'})$ $(1\le i<j\le k)$.

The induced Zappa-Sz\'ep product $\Fth\bowtie G$ is called the {\rm product of self-similar actions} $\{(G, [n_i])\}_{i=1}^k$.
\end{prop}

\begin{proof}
``Only if": If $\theta_{ij}( x^i_\fs, x^j_\ft )= (x^j_{\ft'}, x^i_{\fs'})$, then $x^i_\fs x^j_\ft = x^j_{\ft'} x^i_{\fs'}$. So from (B4) and (B6) one has
\begin{align*}
(g \cdot x_\fs^i) (g \vert_{x_\fs^i} \cdot x_\ft^j)&=g\cdot(x^i_\fs x^j_\ft) = g\cdot(x^j_{\ft'} x^i_{\fs'})=(g \cdot x_{\ft'}^j)(g \vert_{x_{\ft'}^j} \cdot x_{\fs'}^i)\\
g|_{x_\fs^i}|_{ x_\ft^j}&=g|_{x_\fs^ix_\ft^j}=g|_{x_{\ft'}^jx_{\fs'}^i}=g|_{x_{\ft'}^j}|_{ x_{\fs'}^i}
\end{align*}
for all $g\in G$. In particular, \eqref{E:welldef} 
holds true.

``If": In fact, for $g\in G$ and $u_i\in \mathbb{F}_{n_i}^+$ $(1\le i\le k)$, define
\begin{align*}
g\cdot(x_{u_1}^{1}x_{u_2}^{2}\cdots x_{u_k}^{k})
&:=(g\cdot x_{u_1}^{1})(g|_{x_{u_1}^{1}}\cdot x_{u_2}^{2})\cdots
     ({g|_{x_{u_1}^{1}}}|_{x_{u_2}^{2}}\ldots|_{x_{u_{k-1}}^{{k-1}}} \cdot x_{u_k}^{k}),\\
g|_{x_{u_1}^{1}x_{u_2}^{2}\cdots x_{u_k}^{k}}
&:=g|_{x_{u_1}^{1}}|_{x_{u_2}^{2}}\ldots |_{x_{u_{k}}^{k}}.
\end{align*}
Notice that using (B2) and (B8) one can easily see that \eqref{E:welldef} 
holds true for all $g\in G$.

Here we only check Condition (B4) in Definition \ref{D:Z-S}, and similarly for the others.
Clearly it suffices to verify
\[
g\cdot (x_{\fs_1}^{i_1}\cdots x_{\fs_k}^{i_k})
=(g\cdot x_{\fs_1}^{i_1})(g|_{x_{\fs_1}^{i_1}}\cdot x_{\fs_2}^{i_2})\cdots
     ({g|_{x_{\fs_1}^{i_1}}}|_{x_{\fs_2}^{i_2}}\ldots|_{x_{\fs_{k-1}}^{i_{k-1}}} \cdot x_{\fs_k}^{i_k}) 
\]
where all $i_n$'s distinct and $\fs_j\in [m_{i_j}]$.
But this follows from the facts that (B2) and (B8) hold true on each $(G, [m_i])$, and that each word $x_{\fs_1}^{i_1}\cdots x_{\fs_k}^{i_k}$
can be obtained from $x_{\ft_1}^{1}\cdots x_{\ft_k}^{k}$ after finite steps by switching the super indices $i_n$ and $i_{n+1}$ with $i_n>i_{n+1}$ only once at one time.
\end{proof}

\begin{rem}
To see the above proposition is a generalization of  \cite[Proposition 3.10]{BRRW14}, let $k=2$, $\theta:=\theta_{12}$, $x:=x_\fs^1$, $y:=x_\ft^2$,
 $x_{\fs'}^1=\theta_X(x_\fs,y_\ft)$ and $y_{\ft'}^2=\theta_Y(x_\fs,y_\ft)$. Then
 \begin{align*}
&\begin{cases}
g\cdot\theta_Y(x,y)=\theta_Y(g\cdot x, g|_x\cdot y), \\
g_{\theta_Y(x,y)}\cdot\theta_X(x,y)=\theta_X(g\cdot x, g|_x\cdot y)
\end{cases}\\
\Leftrightarrow& (g\cdot\theta_Y(x,y))(g_{\theta_Y(x,y)}\cdot\theta_X(x,y))=\theta_Y(g\cdot x, g|_x\cdot y)\theta_X(g\cdot x, g|_x\cdot y)\\
& (\text{by the unique factorization property of }\Fth)\\
\Leftrightarrow& g\cdot(\theta_Y(x,y)\theta_X(x,y))=(g\cdot x)(g|_x\cdot y)\\
\Leftrightarrow& g\cdot(xy)=(g\cdot x)(g|_x\cdot y).
\end{align*}
\end{rem}

\begin{eg}
\label{Eg:nrLCM}
Let $n_i=n$ for all $1\le i\le k$ and $\theta_{ij}(\fs,\ft)=(\fs,\ft)$ for all $1\le i<j\le k$. Then it is easy to check that $\Fth$ is a
$k$-graph (also see \cite{DY09}). Let $G$ be an \textit{arbitrary} group self-similarly acting on each $\Fn$ in the same way. So if
\[
g \cdot e_\fs^i= e_{\fs_1}^i, \ g \vert_{e_\fs^i}=h,\ h \cdot e_\ft^j=e_{\ft_1}^j,\ h|_{e_\ft^j}=h_1,
\]
then
\[
g \cdot e_\fs^j= e_{\fs_1}^j, \ g \vert_{e_\fs^j}=h,\ h \cdot e_\ft^i=e_{\ft_1}^i,\ h|_{e_\ft^i}=h_1.
\]
Thus
\begin{align*}
(g \cdot e_\fs^i) (g \vert_{e_\fs^i} \cdot e_\ft^j)
&=e_{\fs_1}^ie_{\ft_1}^j=e_{\fs_1}^je_{\ft_1}^i
=(g \cdot e_{\fs}^j)(g \vert_{e_{\ft}^j} \cdot e_{\fs}^i),\\
g|_{e_\fs^i}|_{e_\ft^j}
&=h|_{e_\ft^j}=h_1=h|_{e_\ft^i}=g|_{e_\fs^j}|_{e_\ft^i}.
\end{align*}
It follows from Proposition \ref{P:kZ-S} that one obtains the product $\Fth\bowtie G$ of self-similar actions $(G,[n])$.
\end{eg}

It is worth mentioning that the above $\Fth$ is not a right LCM at all (as $\Fth$ is periodic), and so in this case $\Fth\bowtie G$ is not right LCM.

In the sequel, we exhibit a class of products of self-similar actions satisfying all conditions in Theorem \ref{T:kQ(LambdaxG)}, which plays a vital role in this paper.

\begin{eg}
\label{Eg:podo}
Let $n_1,\ldots, n_k$ be $k$ \textit{arbitrarily} positive integers. For each $1 \leq i \leq k$, let $X_i:=\{x^i_\fs: \fs\in [n_i]\}$, and
let $\bZ$ act on each $X_i$ as an odometer (see Example \ref{Eg:odometer}).  For $1 \leq i<j \leq k$, let $\theta_{ij}:X_i \times X_j \to X_j \times X_i$ be a bijection defined by
\begin{align}
\label{E:odo}
\theta_{ij}(x^i_\fs,x^j_\ft)=(x^j_{\ft'},x^i_{\fs'})\text{ if }\fs+\ft n_i=\ft'+\fs' n_j\  (\fs,\fs'\in [n_i], \ft,\ft'\in[n_j]).
\end{align}
Let
\[
\Fth=\big\langle x_\ft^i:  \ft\in [n_i],\ 1\le i\le k; \, \theta_{ij} \text{ in }\eqref{E:odo}, \, 1\le i<j\le k \big\rangle.
\]
One can easily check that $\Fth$ satisfies the cubic condition, and so $\Fth$ is indeed a single-vertex $k$-graph.
Moreover, the relation \eqref{E:welldef} 
is satisfied.
Then applying Proposition \ref{P:kZ-S} gives a Zappa-Sz\'ep product $\Fth\bowtie\bZ$, which is
the product of odometers $\{(G, [n_i])\}_{i=1}^k$.
\end{eg}

\begin{rem}
\label{R:surj}
By induction, one can check that, given $\mu \in \bF_\theta^+$ and $l \in \mathbb{Z}$, there exists $l' \in \mathbb{Z}$ such that $l' \vert_\mu=l$.
So the restriction map satisfies the condition required in Theorem \ref{T:kQ(LambdaxG)}.
\end{rem}

Due to its importance of the above example and also the natural definition of $\theta$,  the induced Zappa-Sz\'ep product $\Fth\bowtie\bZ$
deserves its own name.

\begin{defn}
\label{D:podo}
The Zappa-Sz\'ep product $\Fth\bowtie\bZ$ given in Example \ref{Eg:podo} is called the \textit{standard product of odometers  $\{(\bZ, [n_i])\}_{i=1}^k$}.
\end{defn}

The following proposition is a generalization of a result from \cite{BRRW14}.

\begin{prop}
\label{P:iffLCM}
Keep the same notation as in Example \ref{Eg:podo}.
Then the following statements are equivalent:
\begin{enumerate}
\item\label{n_1 n_2 dots mut coprime} $n_i$'s are pairwise coprime;
\item\label{existence of st=t's' given s,t'} for $1 \leq i < j \leq k$, given $(\fs,\ft')\in[n_i]\times [n_j] $, there exists a unique pair $(\fs',\ft) \in [n_i]\times [n_j]$ such that
$x^i_{\fs}x^j_{\ft}=x^j_{\ft'}x^i_{\fs'}$;
\item\label{unique CM with degree d(mu) lor d(nu)} any two elements $\mu,\nu \in \bF_\theta^+$ having a right common multiple have a unique right least common multiple with degree $d(\mu) \lor d(\nu)$;
\item\label{F_theta^+ rLCM} $\bF_\theta^+$ is right LCM.
\end{enumerate}
\end{prop}
\begin{proof}
(\ref{n_1 n_2 dots mut coprime})$\Rightarrow$(\ref{existence of st=t's' given s,t'}).
Fix $1 \leq i < j \leq k$ and  $(\fs,\ft')\in[n_i]\times [n_j]$. Assume that $(\fs',\ft), (\fs'',\ft'') \in [n_i]\times [n_j]$ such that
$x^i_{\fs}x^j_{\ft}=x^j_{\ft'}x^i_{\fs'}$ and $x^i_{\fs}x^j_{\ft''}=x^j_{\ft'}x^i_{\fs''}$. Then $\fs+\ft n_i=\ft'+\fs' n_j$ and $\fs+\ft'' n_i=\ft'+\fs'' n_j$. So
$(\ft-\ft'')n_i=(\fs'-\fs'')n_j$. Since $n_i$ and $n_j$ are coprime, $\ft=\ft''$ and $\fs'=\fs''$.

(\ref{existence of st=t's' given s,t'})$\Rightarrow$(\ref{unique CM with degree d(mu) lor d(nu)}) and (\ref{unique CM with degree d(mu) lor d(nu)})$\Rightarrow$(\ref{F_theta^+ rLCM}).
The proof is straightforward.

(\ref{F_theta^+ rLCM})$\Rightarrow$(\ref{n_1 n_2 dots mut coprime}).  To the contrary, suppose that there exist $1 \leq i < j \leq k$ such that $n_i$ and $n_j$ are not coprime.
Let $l:=\gcd(n_i,n_j)$. Then $l >1$. By the definition of $\theta_{ij}$ in \eqref{E:odo}, we have $x^i_{0}x^j_{n_j/l}=x^j_{0}x^i_{n_i/l}$ and
$x^i_{0}x^j_{0}=x^j_{0}x^i_{0}$. We deduce that $x^i_{0}$ and $x^j_{0}$ have right common multiples, but they do not have a right least common multiple.
This contradicts the assumption that $\bF_\theta^+$ is right LCM. Therefore $n_i$'s are pairwise coprime.
\end{proof}

\begin{rem}
\label{L:uLCM}
In order to include more examples, let us emphasize again that $n_i$'s  are arbitrary positive integers.
As shown in Proposition~\ref{P:iffLCM}, $\bF_\theta^+$ is right LCM if and only if $n_i$ are pairwise coprime. Therefore \cite[Theorem 5.2]{BRRW14} only applies  to the case that
$n_i$'s are pairwise coprime. However, with the aid of Theorem \ref{T:kQ(LambdaxG)}, we are still able to simplify $\mathcal{Q}(\bF_\theta^+ \bowtie \mathbb{Z})$ without any conditions for $n_i$'s.
\end{rem}

\begin{thm}
\label{T:universal}
Let $\bF_\theta^+ \bowtie \mathbb{Z}$ be the standard product of odometers  $\{(\bZ$, $[n_i])\}_{i=1}^k$.
Then $\mathcal{Q}(\bF_\theta^+ \bowtie \mathbb{Z})$ is isomorphic to the universal C*-algebra $\mathcal{A}$ generated by a unitary $f$ and a family of isometries $\{g_{x^i_{\fs}}: \fs\in[n_i], 1 \leq i \leq k\}$ satisfying
\begin{enumerate}
\item[(i)]
\label{sum g_xi g_xi*=1} $\sum_{\fs\in [n_i]}g_{x^i_{\fs}}g_{x^i_{\fs}}^*=1$ for all $1 \leq i \leq k$;
\item[(ii)]\label{f g_xi} for $1 \leq i \leq k$,
$f g_{x^i_{\fs}}= \begin{cases}
    g_{x^i_{\fs+1}} &\text{ if } 0 \leq \fs <n_i-1 \\
    g_{x^i_{0}} f &\text{ if } \fs=n_i-1;
\end{cases}$
\item[(iii)]\label{g_xi g_yj}
$g_{x^i_{\fs}}g_{x^i_{\ft}}=g_{x^j_{\ft'}}g_{x^i_{\fs'}}$ whenever $\theta_{ij}(\fs, \ft)=(\ft',\fs')$ for all $1 \leq i <j \leq k$, $ \fs,\fs'\in [n_i] $ and $\ft,\ft'\in[n_j]$.
\end{enumerate}
\end{thm}

\begin{proof}
We adopt the characterization of $\mathcal{Q}(\bF_\theta^+ \bowtie \mathbb{Z})$ from Theorem \ref{T:kQ(LambdaxG)}. Let $\{t_\mu,s_N:\ \mu \in \bF_\theta^+,N \in \mathbb{Z}\}$ be the generators of $\mathcal{Q}(\bF_\theta^+ \bowtie \mathbb{Z})$.

For $N \in \mathbb{Z}$, define $S_N:=f^N$. Clearly $S$ is a unitary representation of $\bZ$ in $\A$.
Define $T_\mt:=1$, and $T_{x^i_{\fs}}:=g_{x^i_\fs}$  for $\fs\in[n_i]$ ($1\le i\le k$).
By (iii), for any word $w=u_1\cdots u_k\in\Fth$ with $u_i\in [n_i]$, we can define an isometry
$T_w:=g_{x^1_{u_1}}\cdots g_{x^k_{u_k}}$. So this yields an isometric representation of $\Fth$ in $\A$.
Then it follows from (i) that $T$ is a $*$-representation of $\Fth$.

For $1 \leq i \leq k$ and $\fs\in [n_i]$, (ii) implies that
\begin{align*}
S_1 T_{x^i_{\fs}}=f g_{x^i_{\fs}}=\begin{cases}
    g_{x^i_{\fs+1}} &\text{ if } 0 \leq \fs <n_i-1 \\
    g_{x^i_{0}} f &\text{ if } \fs=n_i-1;
\end{cases} =T_{1 \cdot x^i_{\fs}}S_{1|_{x^i_{\fs}}}.
\end{align*}
Then one can easily check that Eq.~\eqref{E:ST} holds true.
By the universal property of $\mathcal{Q}(\bF_\theta^+ \bowtie \mathbb{Z})$, there exists a homomorphism $\varphi:\mathcal{Q}(\bF_\theta^+ \bowtie \mathbb{Z}) \to \mathcal{A}$ such that $\varphi(s_N)=S_N$ and $\varphi(t_\mu)=T_\mu$ for all $N \in \mathbb{Z}$ and $\mu \in \bF_\theta^+$.

Conversely, define $F:=s_1$ and $G_{x^i_{\fs}}:=t_{x^i_{\fs}}$ for $\fs\in[n_i]$ ($1 \leq i \leq k$).
Since $t$ is a $*$-representation of $\Fth$, (i) and (iii) automatically hold true.
For $1 \leq i \leq k$ and $\fs\in[n_i]$, it follows from Eq.~\eqref{E:ST} that
\begin{align*}
 FG_{x^i_{\fs}}=s_1 t_{x^i_{\fs}}=t_{1 \cdot x^i_{\fs}} s_{1 \vert_{x^i_{\fs}}}= \begin{cases}
    G_{x^i_{\fs+1}} &\text{ if } 0 \leq \fs <n_i-1 \\
    G_{x^i_{0}} f &\text{ if }\fs=n_i-1,
\end{cases}
\end{align*}
which implies (ii). By the universal property of $\mathcal{A}$, there exists a homomorphism $\pi:\mathcal{A} \to \mathcal{Q}(\bF_\theta^+ \bowtie \mathbb{Z})$ such that
$\pi(f)=F,\pi(g_{x^i_{\fs}})=G_{x^i_{\fs}}$ for all $1\le i\le k$ and $\fs\in[n_i]$.

It now follows easily that $\pi \circ \varphi=\id,\varphi \circ \pi=\id$. Therefore we are done.
\end{proof}

The following properties will be used later.

\begin{cor}\label{C:properties}
Keep the same notation as in Theorem \ref{T:universal}.
Then
\begin{enumerate}
\item
$g_{x^i_{0}}g_{x^j_{0}}=g_{x^j_{0}}g_{x^i_{0}}$ for all $1 \leq i<j \leq k$;
\item
$f^\fs g_{x^i_{0}}=g_{x^i_{\fs}}$ for all $1 \leq i \leq k$ and $\fs\in[n_i]$;
\item
$f^{n_i^l N}g_{x^i_{0}}^l=g_{x^i_{0}}^l f^N$ for all $1 \leq i \leq k, l \ge 0, N\in\bZ$.
\end{enumerate}
\end{cor}

\begin{proof}
The proofs of (1) and (2)  follow directly from Theorem \ref{T:universal}.

Clearly, the identities of (3) hold trivially when either $l=0$ or $N=0$. So we may assume that $l \geq 1$ and $N \ne 0$. Since $f$ is a unitary, it suffices to verify them for $N>0$.
Also, it is easy to see that one only needs to show $f^{n_i^l}g_{x^i_{0}}^l=g_{x^i_{0}}^l f$, and we do it by induction. Property (ii) of Theorem \ref{T:universal} gives
$f^{n_i} g_{x^i_{0}}=g_{x^i_{0}}f$. Suppose that $f^{n_i^l}g_{x^i_{0}}^l=g_{x^i_{0}}^l f$ holds for $l \geq 1$. Then
$f^{n_i^{l+1}}g_{x^i_{0}}^{l+1}=g_{x^i_{0}}^l f^{n_i}g_{x^i_{0}}=g_{x^i_{0}}^{l+1} f$. This finishes the proof.
\end{proof}

\section{$\Q(\Fth\bowtie \bZ)$ via topological $k$-graphs}
\label{S:main}

In this main section, we first construct a class of topological $k$-graphs $\{\Lambda_\bn:\bn\in\bN^k\}$, which is a higher-dimensional generalization of a topological
graphs $\{E_{n,1}:n\in \bN\}$ studied by Katsura
in \cite{Kat08}. We associate to each $\Lambda_{\bn}$ a product system $X(\Lambda_\bn)$ over $\bN^k$.
The first main result here shows that the associated Cuntz-Pimsner C*-algebra $\O_{X(\Lambda_\bn)}$
of $X(\Lambda_\bn)$ is isomorphic to the boundary quotient C*-algebra $\Q(\Fth\bowtie\bZ)$ of the standard product of $k$ odometers (Theorem \ref{T:top2}).
Then, motivated by and with the aid of some results in \cite{Cun08, Kat08, Yam09}, we prove our second main theorem (Theorem \ref{T:simple}) in this section:
$\O_{X(\Lambda_\bn)}$ is simple if and only if $\{\ln n_i: 1\le i\le k\}$ is rationally independent, and is also purely infinite in these cases.
The nuclearity of $\O_{X(\Lambda_\bn)}$ is obtained by applying some results in \cite{CLSV11, Yee07}.
By \cite{Tu99}, $\O_{X(\Lambda_\bn)}$ satisfies the UCT too. 
Combing these two theorems gives a very clear picture on $\Q(\Fth\bowtie \bZ)$ (Theorem \ref{T:simplepure}).
At the end of this section, we also provide some relations between $\Q(\Fth\bowtie\bZ)$ and the C*-algebra $\Q_\bN$ introduced by Cuntz in \cite{Cun08}.

From now on, we only consider the standard product $\Fth\bowtie\bZ$ of the odometers  $\{(\bZ, [n_i])\}_{i=1}^k$.
For our convenience, we use the notation
\[
\textbf{1}:=(1,\ldots, 1),\ \bn:=(n_1,\ldots, n_k),\  \bn^p:=\prod_{i=1}^k n_i^{p_i} \ (p\in \bN^k).
\]

\subsection{Realizing $\Q(\Fth\bowtie \bZ)$ as topological $k$-graph C*-algebras}
\label{SS:Lambdan}
In this subsection, we first construct a class of topological $k$-graphs, whose C*-algebras will be shown to be isomorphic to $\Q(\Fth\bowtie\bZ)$.

\begin{defn}
\label{D:top2}
Let $\Lambda_\bn:=\bigsqcup_{p \in \mathbb{N}^k}\mathbb{T}$  be a topological $k$-graph constructed as follows:
$\Lambda_\bn^0:=\mathbb{T} \times \{0\}$. Given $(z,p) \in \Lambda_\bn$, let
\[
r(z,p):=(z,0),\
s(z,p):=(z^{\bn^p},0),\ d(z,p):=p.
\]
For $(z,p),(w,q) \in \Lambda_\bn$ with $s(z,p)=r(w,q)$, define
\[
(z,p)\cdot (w,q):=(z,p+q).
\]
\end{defn}

One can also describe $\Lambda_\bn$ as follows:
\begin{align*}
 \Lambda_\bn^{e_i}:=\bT,\ r(z,e_i):=(z,0),\ s(z,e_i):=(z^{n_i},0)\ (z\in \bT, 1\le i\le k).
\end{align*}
The commuting squares of $\Lambda_\bn$ are given by
\[
(z,e_i)(z^{n_i},e_j)=(z,e_j)(z^{n_j},e_i) \foral z\in \bT\text{ and }1\le i\ne j\le k.
\]

Thus it is not hard to see that the graph $\Lambda_\bn$ is a $k$-dimensional generalization of Katsura's topological graph
$E_{n,1}$ in \cite{Kat08}, which can also be obtained as \hskip .5cm $\xymatrix{\bullet\ar@(lu,ld)}\times_{n,1}\bT$.
In fact, let $\Lambda$ be the single-vertex $k$-graph with
one edge for each degree $e_i$. Then one could thought of $\Lambda_\bn$ as $\Lambda\times_{\bn, \mathbf{1}} \bT$.

\begin{rem}
$\Lambda_\bn$ is indeed a topological $k$-graph (for $k\ge 1$). In fact, it suffices to verify that $\Lambda_\bn$ satisfies the cubic condition for $k\ge 3$.
To this end, consider $\lambda=(z,e_i)(z^{n_i},e_j)(z^{n_in_j}, e_\fk)$ of degree $(1,1,1)$.
Then
\begin{align*}
\lambda
&=(z,e_j)(z^{n_j},e_i) (z^{n_in_j}, e_\fk)
 =(z,e_j)(z^{n_j},e_\fk) (z^{n_jn_\fk}, e_i)\\
&=(z,e_\fk)(z^{n_\fk},e_j) (z^{n_jn_\fk}, e_i)\\
&=(z,e_i)(z^{n_j},e_\fk) (z^{n_in_\fk}, e_j)
=(z,e_\fk)(z^{n_\fk},e_i) (z^{n_in_\fk}, e_j).
\end{align*}
This exactly says that the cubic condition holds true.
\end{rem}

\begin{lem}
\label{L:CPLambda}
There is a Cuntz-Pimsnder covariant representation of $X(\Lambda_\bn)$ in $\Q(\Fth\bowtie  \bZ)$.
\end{lem}

\begin{proof}

In the sequel, we adopt the characterization of $\Q(\Fth\bowtie  \bZ)$ from Theorem \ref{T:universal}.

Let $\iota:\Lambda_\bn^0 \to \mathbb{C}$ be the embedding map.
Since $f$ is a unitary in $\Q(\Fth\bowtie  \bZ)$, there exists a homomorphism $\psi_0:\rC(\Lambda^0) \to \Q(\Fth\bowtie  \bZ)$ such that $\psi_0(\iota)=f$.

Fix $0 \neq p \in \mathbb{N}^k$. For $l \in \mathbb{Z}$, define $\gamma_l:\Lambda_\bn^{p} \to \mathbb{C}$ by
\[
\gamma_l(z,p):=z^l \qforal  z \in \mathbb{T}.
\]
Let $\mathcal{F}:=\{\gamma_l\}_{l \in \mathbb{Z}}$ and $\mathcal{G}:=\{\iota\}$.
 It is straightforward to see that $\Vert \gamma\Vert_{X_p(\Lambda_\bn)} \leq \sqrt{\bn^p}\Vert \gamma \Vert_{\sup}$ for all $\gamma \in X_p(\Lambda_\bn)$.
 By the Stone-Weierstrass theorem, the linear span of $\mathcal{F}$ is dense in $X_p(\Lambda_\bn)$. It is also straightforward to see that $\mathcal{G}$ generates $\rC(\Lambda^0)$.
Furthermore, $\mathcal{G}\cdot\mathcal{F} \subset \mathcal{F}$.

\smallskip
\textbf{Step 1.}
Construct a linear map $\psi_p: X_p(\Lambda_\bn)\to  \Q(\Fth\bowtie  \bZ)$ such that  $(\psi_p, \psi_0)$ is a representation of $X_p(\Lambda_\bn)$.

Let $i_0:=\min\{1\le i\le k: p_i\ne 0\}$ and $\bg_0^p:=\prod_{i=i_0+1}^{k}g_{x^i_{0}}^{p_i}$.

Clearly, $\mathcal{F}$ is linear independent. Define a linear map $\psi_p:\spn \mathcal{F} \to  \Q(\Fth\bowtie  \bZ)$ by
\begin{align}
\label{E:psigamma}
\psi_p(\gamma_{\fs+n_{i_0}l})
=\bn^{\frac{p}{2}} g_{x^{i_0}_\fs}f^l g_{x^{i_0}_0}^{p_{i_0}-1} \bg_0^p \qforal s\in[n_{i_0}] \text{ and } l \in \mathbb{Z}.
\end{align}

Then we have
\[
\psi_p(\iota \cdot \gamma_{\fs+n_{i_0}l})=\psi_0(\iota)\psi_p(\gamma_{\fs+n_{i_0}l})\qforal \fs\in[n_{i_0}] \text{ and }l \in \mathbb{Z}.
\]
This is done by the following calculations:
For $0 \leq \fs <n_{i_0}-1$ and $l \in \mathbb{Z}$
\begin{align*}
\psi_p(\iota \cdot \gamma_{\fs+n_{i_0}l})
&=\psi_p(\gamma_{\fs+1+n_{i_0}l})
   =\bn^{\frac{p}{2}} g_{x^{i_0}_{\fs+1}}f^l g_{x^{i_0}_0}^{p_{i_0}-1} \bg_0^p\\
&=\bn^{\frac{p}{2}} fg_{x^{i_0}_{\fs}}f^l g_{x^{i_0}_0}^{p_{i_0}-1} \bg_0^p
   =\psi_0(\iota)\psi_p(\gamma_{\fs+n_{i_0}l}),
\end{align*}
and  by Corollary \ref{C:properties}
\begin{align*}
\psi_p(\iota \cdot \gamma_{n_{i_0}-1+n_{i_0}l})
&=\psi_p(\gamma_{n_{i_0}(l+1)})
  =\bn^{\frac{p}{2}} g_{x^{i_0}_0}f^{l+1} g_{x^{i_0}_0}^{p_{i_0}-1} \bg_0^p\\
 &=\bn^{\frac{p}{2}} fg_{x^{i_0}_{n_{i_0}-1}}f^{l} g_{x^{i_0}_0}^{p_{i_0}-1} \bg_0^p
   =\psi_0(\iota)\psi_p(\gamma_{n_{i_0}-1+n_{i_0}l}).
\end{align*}

Now for $\fs,\fs' \in[n_{i_0}]$ and $l,l' \geq 0$ with $\fs+n_{i_0}l,\fs'+n_{i_0}l' \in[\bn^p]$, we claim that
\begin{align}
\nonumber
\psi_p&(\gamma_{\fs+n_{i_0}l+\bn^p m})^*\psi_p(\gamma_{\fs'+n_{i_0}l'+\bn^p m'})\\
\label{E:psi}
&=\psi_0\Big(\langle\gamma_{\fs+n_{i_0}l+\bn^p m},\gamma_{\fs'+n_{i_0}l'+\bn^p m'}\rangle_{\rC(\Lambda_\bn^0)}\Big)
\end{align}
for all $m,m'\in\bZ$.

On the one hand, by a direct calculation, one has
\begin{align*}
\nonumber
&\psi_0\Big(\langle\gamma_{\fs+n_{i_0}l+\bn^p m},\gamma_{\fs'+n_{i_0}l'+\bn^p m'}\rangle_{\rC(\Lambda_\bn^0)}(z,0)\Big)\\
&=\psi_0\Big(\sum_{w^{\bn^p}=z} w^{\fs'-\fs+n_{i_0}(l'-l)+\bn^p (m'-m)}\Big)\\
\nonumber
&=\delta_{\fs,\fs'}\delta_{l,l'}\bn^p \psi_0(z^{m'-m})\\
&=\delta_{\fs,\fs'}\delta_{l,l'}\bn^p f^{m'-m}.
\end{align*}

On the other hand, repeatedly applying Corollary \ref{C:properties} (3) we have
\begin{align}
\nonumber
&\psi_p(\gamma_{\fs+n_{i_0}l+\bn^p m})^*\psi_p(\gamma_{\fs'+n_{i_0}l'+\bn^p m'})\\
\nonumber
&\stackrel{\eqref{E:psigamma}}
=\bn^p  \Big(\bg_0^p\Big)^* \Big(g_{x^{i_0}_0}^{p_{i_0}-1}\Big)^*\Big(f^{l+\frac{\bn^p}{n_{i_0}}m}\Big)^*
     \Big(g_{x^{i_0}_{\fs}}\Big)^* g_{x^{i_0}_{\fs'}}f^{l'+\frac{\bn^p}{n_{i_0}}m'} g_{x^{i_0}_0}^{p_{i_0}-1} \bg_0^p\\
\nonumber
&=\delta_{\fs,\fs'} \bn^p  \Big(\bg_0^p\Big)^* \Big(g_{x^{i_0}_0}^{p_{i_0}-1}\Big)^*f^{l'-l+\frac{\bn^p}{n_{i_0}}(m'-m)} g_{x^{i_0}_0}^{p_{i_0}-1} \bg_0^p
  \  (\text{as } g_{x^{i_0}_\fs}^*g_{x^{i_0}_{\fs'}}=\delta_{\fs,\fs'})\\
\nonumber
&=\delta_{\fs,\fs'} \bn^p  \Big(\bg_0^p\Big)^* \Big(g_{x^{i_0}_0}^{p_{i_0}-1}\Big)^*f^{l'-l} g_{x^{i_0}_0}^{p_{i_0}-1}  f^{\check{\bn}^{\check{p}}(m'-m)}  \bg_0^p\\
\label{E:psi_p}
&=\delta_{\fs,\fs'} \bn^p  \Big(\bg_0^p\Big)^* \Big(g_{x^{i_0}_0}^{p_{i_0}-1}\Big)^*f^{l'-l} g_{x^{i_0}_0}^{p_{i_0}-1} \bg_0^p f^{m'-m},
\end{align}
where $\check{\bn}^{\check{p}}:=\Pi_{i\ne i_0}n_i^{p_i}=\Pi_{i=i_0+1}^k n_i^{p_i}$ as $p_i=0$ for all $i<i_0$.

If $l=l'$, then it follows from \eqref{E:psi_p} that
\begin{align*}
\psi_p(\gamma_{\fs+n_{i_0}l+\bn^p m})^*\psi_p(\gamma_{\fs'+n_{i_0}l'+\bn^p m'})
=\delta_{\fs,\fs'}\bn^pf^{m'-m}.
\end{align*}
If $l\ne l'$, then
 repeatedly applying Corollary \ref{C:properties} to \eqref{E:psi_p}, we obtain
\[
\psi_p(\gamma_{\fs+n_{i_0}l+\bn^p m})^*\psi_p(\gamma_{\fs'+n_{i_0}l'+\bn^p m'})=0.
\]
Thus
\[
\psi_p(\gamma_{\fs+n_{i_0}l+\bn^p m})^*\psi_p(\gamma_{\fs'+n_{i_0}l'+\bn^p m'})=\delta_{\fs,\fs'}\delta_{l,l'}\bn^p f^{m'-m}.
\]
Therefore, we prove \eqref{E:psi}.

By Lemma \ref{L:Toep}, $\psi_p$ can be uniquely extended to a bounded linear map on $X_p(\Lambda_\bn)$, which is still denoted by $\psi_p$.
From above, we have shown that $(\psi_p,\psi_0)$ is a representation of $X_p(\Lambda_\bn)$.

\smallskip
\textbf{Step 2.} We show that $\{(\psi_p, \psi_0):p\in \bN^k\}$ satisfies Condition (T2) of Definition \ref{D:Toep}.

For this, fix $p,q \in \mathbb{N}^k$. Let $i_0:=\min\{1\le i\le k: p_i\ne 0\}$ and $i_0':=\min\{1\le i\le k: q_i\ne 0\}$. Without loss of generality, let us assume that $i_0\le i_0'$.
Repeatedly applying Corollary \ref{C:properties} yields
\begin{align}
\nonumber
&\psi_p(\gamma_{\fs+n_{i_0}l})\psi_q(\gamma_{\fs'+n_{i_0'}l'})\\
\nonumber
&\stackrel{\eqref{E:psigamma}}=\bn^{\frac{p+q}{2}} g_{x^{i_0}_\fs}f^l \Big(g_{x^{i_0}_0}\Big)^{p_{i_0}-1}\bg_0^p
     g_{x^{i_0'}_{\fs'}}f^{l'} \Big(g_{x^{i_0'}_0}\Big)^{q_{i_0'}-1}\bg_0^q\\
\nonumber
&=\bn^{\frac{p+q}{2}} g_{x^{i_0}_\fs}f^l \Big(g_{x^{i_0}_0}\Big)^{p_{i_0}-1} \bg_0^p
     f^{\fs'}g_{x^{i_0'}_{0}}f^{l'} \Big(g_{x^{i_0'}_0}\Big)^{q_{i_0'}-1} \bg_0^q\\
\nonumber
&=\bn^{\frac{p+q}{2}} g_{x^{i_0}_\fs}f^l \Big(g_{x^{i_0}_0}\Big)^{p_{i_0}-1} f^{(\prod_{i=i_0+1}^k n_i^{p_i})\fs'}\bg_0^p
     g_{x^{i_0'}_{0}}f^{l'} \Big(g_{x^{i_0'}_0}\Big)^{q_{i_0'}-1} \bg_0^q\\
\nonumber
 &=\bn^{\frac{p+q}{2}} g_{x^{i_0}_\fs}f^l \Big(g_{x^{i_0}_0}\Big)^{p_{i_0}-1} f^{(\prod_{i=i_0+1}^k n_i^{p_i})(\fs'+n_{i_0'}l')}
      \bg_0^p
      \Big(g_{x^{i_0'}_0}\Big)^{q_{i_0'}} \bg_0^q\\
\label{E:psip}
  &=\bn^{\frac{p+q}{2}} g_{x^{i_0}_\fs} f^{l+(\prod_{i=i_0+1}^k n_i^{p_i})n_{i_0}^{p_{i_0}-1}(\fs'+n_{i_0'}l')}
      \Big(g_{x^{i_0}_0}\Big)^{p_{i_0}-1}\bg_0^p
      \Big(g_{x^{i_0'}_0}\Big)^{q_{i_0'}} \bg_0^q.
\end{align}
Notice that
\begin{align*}
\bg_0^{p+q}=\begin{cases}
\bg_0^p\bg_0^q &\text{ if } i_0=i_0'\\
\bg_0^p\Big(g_{x^{i_0'}_0}\Big)^{q_{i_0'}} \bg_0^q &\text{ if } i_0<i_0'.
\end{cases}
\end{align*}
Then continuing \eqref{E:psip} gives
\begin{align*}
&\psi_p(\gamma_{\fs+n_{i_0}l})\psi_q(\gamma_{\fs'+n_{i_0'}l'})\\
&=\begin{cases}
\bn^{\frac{p+q}{2}} g_{x^{i_0}_\fs} f^{l+(\prod_{i=i_0+1}^k n_i^{p_i})n_{i_0}^{p_{i_0}-1}(\fs'+n_{i_0'}l')}
      \Big(g_{x^{i_0}_0}\Big)^{p_{i_0}+q_{i_0}-1}\bg_0^{p+q}& \text{ if }i_0=i_0'\\
\bn^{\frac{p+q}{2}} g_{x^{i_0}_\fs}f^{l+\frac{\bn^p}{n_{i_0}}({\fs'+n_{i_0'}l'})} \Big(g_{x^{i_0}_0}\Big)^{p_{i_0}-1} \bg_0^{p+q}&\text{ if }i_0<i_0'\\
\end{cases}\\
&=\psi_{p+q}\Big(\gamma_{\fs+n_{i_0}\big(l+\frac{\bn^p}{n_{i_0}}({\fs'+n_{i_0'}l'})\big)}\Big)\\
&=\psi_{p+q}\Big(\gamma_{\fs+n_{i_0}l} \diamond \gamma_{\fs'+n_{i_0'}l'}\Big).
\end{align*}
Here the last ``=" above holds true due to the following:
\begin{align*}
\gamma_{\fs+n_{i_0}l} \diamond \gamma_{\fs'+n_{i_0'}l'}\Big(z,p+q\Big)
&=\gamma_{\fs+n_{i_0}l} \diamond \gamma_{\fs'+n_{i_0'}l'}\Big((z,p)(z^{\bn^p,q})\Big)\\
&=\gamma_{\fs+n_{i_0}l}\Big((z,p)\Big) \gamma_{\fs'+n_{i_0'}l'}\Big((z^{\bn^p},q)\Big)\\
&=z^{\fs+n_{i_0}\big(l+\frac{\bn^p}{n_{i_0}}({\fs'+n_{i_0'}l'})\big)}\\
&=\gamma_{\fs+n_{i_0}\big(l+\frac{\bn^p}{n_{i_0}}({\fs'+n_{i_0'}l'})\big)}\Big((z,p+q)\Big).
\end{align*}

Thus far, we have finished the proof of \textbf{Step 2}.

Therefore, by piecing $\{\psi_p\}_{p \in \mathbb{N}^k}$ together we get a representation $\psi:X(\Lambda_\bn) \to \Q(\Fth\bowtie  \bZ)$.

\smallskip
\textbf{Step 3.} We prove that $\psi$ is Cuntz-Pimsner covariant.

By Lemma \ref{L:CP}, it suffices to show that $(\psi_{e_i},\psi_0)$ ($1\le i\le k$) are Cuntz-Pimsner covariant.

Notice that a simple calculation shows that
\[
\left\langle \frac{\gamma^m}{\sqrt{\bn^{p}}}, \frac{\gamma^{m'}}{\sqrt{\bn^{p}}}\right\rangle
=
\begin{cases}
\iota^N & \text{ if } m'-m=\bn^pN \text{ for }N\in  \bZ\\
0&\text{if }m'-m\not\in \bn^p \bZ.
\end{cases}
\]
Then one can obtain that
\[
\phi_{e_i}(\iota)=\sum_{\fs\in [n_i]}\Theta_{\frac{\gamma_{\fs+1}}{\sqrt n_i}, \frac{\gamma_\fs}{\sqrt n_i}}.
\]
Hence
\begin{align*}
\psi_{e_i}^{(1)}(\phi_{e_i}(\iota))
&=\sum_{\fs\in [n_i]}\psi_{e_i}\left(\frac{\gamma_{\fs+1}}{\sqrt n_i}\right)\psi_{e_i}\left(\frac{\gamma_\fs}{\sqrt n_i}\right)^*
\\&=\sum_{\fs\in [n_i]} \psi_0(\iota)\psi_{e_i}\left(\frac{\gamma_\fs}{\sqrt n_i}\right)\psi_{e_i}\left(\frac{\gamma_\fs}{\sqrt n_i}\right)^*
\\&=\sum_{\fs\in [n_i]}fg_{x^i_\fs}g_{x^i_\fs}^*
\\&=f \ (\text{by Theorem~\ref{T:universal} (i)})
\\&=\psi_0(\iota).
\end{align*}
By Lemma \ref{L:Toep}, $(\psi_{e_i},\psi_0)$ is Cuntz-Pimsner covariant for $1\le i\le k$.
\end{proof}

The following theorem is inspired by \cite{Kat08}.

\begin{thm}
\label{T:top2}
Let $\Lambda_\bn$ be the topological $k$-graph constructed in Definition \ref{D:top2}, and $X(\Lambda_\bn)$ be the product system associated to
$\Lambda_\bn$.
Then $\mathcal{Q}(\bF_\theta^+ \bowtie \mathbb{Z})$ is isomorphic to $\mathcal{O}_{X(\Lambda_\bn)}$.
\end{thm}

\begin{proof}
As before, denote by $\iota:\Lambda_\bn^0 \to \mathbb{C}$ the embedding map. To simplify our writing, in what follows, denote by $\jmath:X(\Lambda_\bn) \to \mathcal{O}_{X(\Lambda_\bn)}$ the universal Cuntz-Pimsner covariant representation of $X(\Lambda_\bn)$ satisfying that $\jmath$ generates $\mathcal{O}_{X(\Lambda_\bn)}$.
Let $\psi: X(\Lambda_\bn)\to \Q(\Fth\bowtie\bZ)$ be the Cuntz-Pimsner covariant representation constructed in the proof of Lemma \ref{L:CPLambda}.
Then there exists a unital homomorphism $\varphi:\mathcal{O}_{X(\Lambda_\bn)} \to \Q(\Fth\bowtie\bZ)$ such that $\varphi \circ \jmath=\psi$.

Conversely, define
\[
I:=\jmath_{0}(1_{\rC(\Lambda_\bn^0)})\text{ and }F:=\jmath_{0}(\iota).
\]
Then $I$ is the identity of $\mathcal{O}_{X(\Lambda_\bn)}$ and $F$ is a unitary in $\mathcal{O}_{X(\Lambda_\bn)}$.
For $1\le i\le k$ and $\fs \in [n_i]$, let $\xi^i_\fs:\Lambda_\bn^{e_i} \to \mathbb{C}$ be
the function $\xi^i_\fs(z,e_i):=z^\fs /\sqrt{n_i}$ for all $z \in \mathbb{T}$, and define
\[
G_{x^i_\fs}:=\jmath_{e_i}(\xi^i_\fs);
\]
For $(z,0) \in \Lambda_\bn^0$, we have
\begin{align*}
\langle \xi^i_\fs,\xi^i_\fs \rangle_{\rC(\Lambda_\bn^0)}(z,0)
&=\sum_{\{(w,e_i) \in \Lambda_\bn^{e_i}:w^{n_i}=z\}}\vert \xi^i_\fs(w,e_i)\vert^2\\
&=\sum_{\{(w,e_i) \in \Lambda_\bn^{e_i}:w^{n_i}=z\}}\frac{1}{n_i}=1.
\end{align*}
So
\[
G_{x^i_\fs}^*G_{x^i_\fs}=\jmath_{0}(\langle \xi^i_\fs, \xi^i_\fs\rangle)=I.
\]
Hence $G_{x^i_\fs}$ is an isometry in $\mathcal{O}_{X(\Lambda_\bn)}$.

For $x \in \rC(\Lambda_\bn^{e_i})$ and $(z,e_i) \in \Lambda_\bn^{e_i}$, we have
\begin{align*}
\sum_{\fs\in[n_i]}\Theta_{\xi^i_\fs,\xi^i_\fs}(x)(z,e_i)&=\sum_{\fs\in[n_i]}\xi^i_\fs(z,e_i)\langle \xi^i_\fs,x\rangle_{\rC(\Lambda_\bn^0)}(z^n,0)
\\&=\sum_{\fs\in[n_i]}\frac{z^\fs}{\sqrt n_i}\left(\sum_{\{w \in \mathbb{T}:w^{n_i}=z^{n_i}\}}\frac{\overline{w}^\fs}{\sqrt n} x(w,e_i)\right)
\\&=\frac{1}{n_i}\sum_{\{w \in \mathbb{T}:w^{n_i}=z^{n_i}\}}\sum_{\fs\in[n_i]}z^\fs \overline{w}^\fs x(w,e_i)
\\&=x(z,e_i)
\\&=\phi_{e_i}(1_{\rC(\Lambda_\bn^0)})(x)(z,e_i),
\end{align*}
where the above 4th ``=" holds true because $\sum_{\fs\in[n_i]}z^\fs \overline{w}^\fs=0$ unless $w=z$.
So
\[
\sum_{\fs\in[n_i]}\Theta_{\xi^i_\fs,\xi^i_\fs}=\phi_{e_i}(1_{\rC(\Lambda_\bn^0)}).
\]
Since $\jmath$ is Cuntz-Pimsner covariant, we obtain
\begin{align*}
\sum_{\fs\in[n_i]}G_{x^i_\fs}G_{x^i_\fs}^*
&=\sum_{\fs\in[n_i]}\jmath_{e_i}(\xi^i_\fs)(\jmath_{e_i}(\xi^i_\fs))^*
=\jmath_{e_i}^{(1)}\Big(\sum_{\fs\in[n_i]}\Theta_{\xi^i_\fs,\xi^i_\fs}\Big)\\
&=\jmath_{e_i}^{(1)}\Big(\phi_{e_i}(1_{\rC(\Lambda_\bn^0)})\Big)
   =\jmath_{0}(1_{\rC(\Lambda_\bn^0)})=I.
\end{align*}

For $0\le \fs< n_i-1$ and $(z,e_i) \in \Lambda_\bn^{e_i}$, we have
\[
(\iota \cdot \xi^i_\fs)(z,e_i)=\frac{z z^\fs}{\sqrt n_i}=\frac{z^{\fs+1}}{\sqrt n_i}=\xi_{\fs+1}(z,e_i).
\]
So
\[
FG_{x^i_\fs}=G_{x^i_{\fs+1}}.
\]
For $\fs=n_i-1$, we compute that
\[
(\iota \cdot \xi^i_{n_i-1})(z,e_i)=\iota(z,0)\xi^i_{n_i-1}(z,e_i)=\frac{z^{n_i}}{\sqrt{n_i}}=\xi_0^i(z,e_i)\iota(z^{n_i},0)=(\xi_0^i \cdot \iota)(z,e_i).
\]
So
\[
FG_{x^i_{n_i-1}}=G_{x^i_0}F.
\]


Observe that
\[
(z,e_i+e_j)=(z,e_i)(z^{n_i},e_j)=(z,e_j)(z^{n_j},e_i)
\]
for all $z\in \bT$ and $1\le i< j\le k$.
Then for $\fs,\fs' \in [n_i]$ and $\ft,\ft' \in [n_j]$ satisfying that $\fs+\ft n_i=\ft'+\fs' n_j$, we have
\begin{align*}
(\xi^i_\fs \diamond \xi^j_\ft)(z,(e_i+e_j))&=\xi^i_\fs(z,e_i)\xi^j_\ft(z^{n_i},e_j)
\\&=\frac{z^\fs}{\sqrt n_i}\frac{z^{\ft n_i}}{\sqrt n_j}
     =\frac{z^{\ft'}}{\sqrt n_j}\frac{z^{\fs' n_j}}{\sqrt n_i}
\\&=\xi^j_{\ft'}(z,e_j)\xi^i_{\fs'}(z^{n_j},e_i)
\\&=(\xi^j_{\ft'}\diamond\xi^i_{\fs'})(z,(e_j+e_i)).
\end{align*}
So $\xi^i_\fs \diamond \xi^j_\ft=\xi^j_{\ft'}\diamond\xi^i_{\fs'}$. By Condition (T2) of Definition \ref{D:Toep}, one has
\[
G_{x^i_\fs}G_{x^j_\ft}=G_{x^{j}_{\ft'}}G_{x^i_{\fs'}}.
\]

Therefore $\{F, G_{x^i_\fs}:\fs\in[n_i], 1\le i\le k\}$ satisfy Conditions (i)--(iii) of Theorem \ref{T:universal}. By Theorem \ref{T:universal}, there exists a unital homomorphism $\pi: \mathcal{Q}(\bF_\theta^+ \bowtie \mathbb{Z}) \to \mathcal{O}_{X(\Lambda_\bn)}$ such that
$\pi(f)=F$ and $\pi(g_{x^i_\fs})=G_{x^i_\fs}$ for all $\fs \in [n_i]$ and $1\le i\le k$.

\smallskip
Since $\ca(\jmath(X(\Lambda_\bn)))=\ca(\{\jmath_{0}(X(\Lambda_\bn)_0),\jmath_{e_i}(X(\Lambda_\bn)_{e_i}):1\le i\le k\})$,
it is straightforward to see that $\pi \circ \varphi=\id, \varphi \circ \pi=\id$. Therefore we are done.
\end{proof}

\subsection{Nuclearity, simplicity and pure infiniteness of $\Q(\Fth\bowtie\bZ)$}
In this subsection, we investigate the conditions under which $\Q(\Fth\bowtie\bZ)$ is nuclear, simple, and purely infinite.

It is necessary to recall the following definitions from \cite{Yam09}.

\begin{defn}
A topological $k$-graph $\Lambda$ is said to satisfy \emph{Condition~(A)} if for any $v \in \Lambda^0$ and for any open neighborhood $V$ of $v$, there exist $v' \in V$ and $\mu \in v'\Lambda^\infty$ such that $\sigma^{p}(\mu) \neq \sigma^{q}(\mu)$ whenever $p \neq q \in \mathbb{N}^k$.
\end{defn}

\begin{defn}
Let $\Lambda$ be a regular topological $k$-graph. For $v \in \Lambda^0$ and for $\mu \in v\Lambda^\infty$, denote by
\[
\mathrm{Orb}^+(v):=r(s^{-1}(v))\text{ and }\mathrm{Orb}(v,\mu):= \bigcup_{n \in \mathbb{N}^k}\mathrm{Orb}^+(\mu(n,n)).
\]
\end{defn}

\begin{defn}
Let $\Lambda$ be a regular topological $k$-graph, and let $V$ be a nonempty precompact open subset of $\Lambda^0$. Then $V$ is said to be \emph{contracting} if there exist finitely many nonempty open subsets $U_i \subset \Lambda^{p_i}$, where $i=1,\dots,l, p_i \in \mathbb{N}^k \setminus \{0\}$, such that
\begin{itemize}
\item[(1)] $r(U_i) \subset V$ for all $1 \leq i \leq l$;
\item[(2)] $\mu(0,p_i \land p_j) \neq \nu(0,p_i \land p_j)$ for all $1 \leq i \neq j \leq l, \mu \in U_i, \nu \in U_j$;
\item[(3)] $\overline{V} \subsetneq \bigcup_{i=1}^{l}s(U_i)$.
\end{itemize}
Furthermore, $\Lambda$ is said to be \emph{contracting} if there exists $v \in \Lambda^0$ such that $\overline{\mathrm{Orb}^+(v)}=\Lambda^0$ and any open neighborhood of $v$ contains an open contracting set.
\end{defn}

\begin{rem}
In order to pursue the simplicity condition of $\Q(\Fth\bowtie\bZ)$, we wish to apply \cite[Theorems~4.7]{Yam09}. However, it was pointed out by Nicolai Stammeier that there is a flaw in the proof of \cite[Theorems~4.7]{Yam09}.
Fortunately, we are able to provide it an alternative proof when $1 \leq k <\infty$
(see Theorem \ref{T:4.7} below and its proof) by invoking the work of Brown-Clark-Farthing-Sims \cite{BCFS14} and the work of Yeend \cite{Yee07}.

For this,  we need to exploit the groupoid C*-algebra technique, which can be referred to \cite{Ren80}. In the sequel, we give a very sketchy introduction to the boundary path groupoid arising from a regular topological $k$-graph (see \cite{Yee07}).

Let $1 \leq k <\infty$ and $\Lambda$ be a regular topological $k$-graph. By recalling the construction of \cite{Yee07}, we get the set of \emph{boundary path} $\partial \Lambda=\Lambda^\infty$, which is endowed with the topology generated by the basic open sets $Z(U):=\{x \in \Lambda^\infty:x(0,n) \in U\}$ where $U$ is an open subset of $\Lambda^n$ for some $n \in \mathbb{N}^k$. The \emph{boundary path groupoid} $\G_\Lambda$ of $\Lambda$ is defined by $\mathcal{G}_\Lambda=\{(x,p-q,y) \in \partial \Lambda \times \mathbb{Z}^k \times \partial \Lambda:\sigma^p(x)=\sigma^q(x)\}$, which is endowed with the topology generated by the basic open sets $Z(U,V):=\{(x,p-q,y) \in U \times \mathbb{Z}^k \times V: \sigma^p(x)=\sigma^q(x)\}$ where $U$ and $V$ are
open in $\Lambda^p$ and $\Lambda^q$ respectively.
\end{rem}

\begin{thm}[{\cite[Theorems~4.7 and 4.13]{Yam09}}]
\label{T:4.7}
Let $1 \leq k <\infty$. Let $\Lambda$ be a regular topological $k$-graph. Suppose $\Lambda$ satisfies Condition~(A) and $\overline{\mathrm{Orb}(v,\mu)}=\Lambda^0$  for any $v \in \Lambda^0$ and $\mu \in v \Lambda^\infty$. Then $\mathcal{O}_{X(\Lambda)}$ is simple. Furthermore, suppose that $\Lambda$ is contracting. Then $\mathcal{O}_{X(\Lambda)}$ is purely infinite.
\end{thm}


\begin{proof}
\cite[Theorem~6.8]{Yee07} yields that $\mathcal{G}_\Lambda$ is amenable. So $\ca(\mathcal{G}_\Lambda)=\ca_r(\mathcal{G}_\Lambda)$ is nuclear. By \cite[Theorem 5.20]{CLSV11} and \cite[Corollary~5.2]{SY10}, $\ca(\mathcal{G}_\Lambda) \cong \mathcal{O}_{X(\Lambda)}$. So we must show that $\ca(\G_{\Lambda})$ is simple.

By \cite[Theorem~5.2]{Yee07}, $\mathcal{G}_\Lambda$ is topologically principal. Let $D$ be a non-empty open invariant subset of $\mathcal{G}_\Lambda^0$. Suppose $D \neq \mathcal{G}_\Lambda^0$, for a contradiction. Then there exists $x \in \mathcal{G}_\Lambda^0 \setminus D$, and so $\overline{\mathrm{Orb}(x(0,0),x)}=\Lambda^0$ by our assumption. Since $D$ is open, take $n \in \mathbb{N}^k$ and a non-empty open subset $U$ of $\Lambda^n$ satisfying $Z(U) \subset D$. Then there exist $\mu,\nu \in \Lambda,m \in \mathbb{N}^k$ such that $\mu \in U,s(\mu)=r(\nu),s(\nu)=x(m,m)$. So $y:=\mu\nu\sigma^m(x) \in D$ and $(x,m-(n+d(\nu)),y) \in \mathcal{G}_\Lambda$. Since $D$ is invariant, one has $x \in D$. This is a contradiction.
Hence $D=\mathcal{G}_\Lambda^0$. Therefore $\mathcal{G}_\Lambda$ is minimal. By \cite[Theorem~5.1]{BCFS14}, $\ca(\mathcal{G}_\Lambda)$ is simple.
\end{proof}

Before giving our main results, two lemmas first.

\begin{lem}
\label{L:aper}
Let $\Lambda_\bn$ be the topological $k$-graph constructed in Definition \ref{D:top2}. If $\{\ln {n_i}\}_{1\le i\le k}$ is rationally independent,
then $\Lambda_\bn$ satisfies Condition (A).
\end{lem}

\begin{proof}
Since $\{\ln {n_i}\}_{1\le i\le k}$ is rationally independent, we have $\bn^p \neq \bn^q$ for all $p \neq q \in \mathbb{N}^k$. Fix $(z,0) \in \Lambda_\bn^0$ and an open neighborhood $V$ of $z$.
Pick up $w \in V$ such that $w=e^{2\pi i \theta}$ with $\theta \in (0,1) \setminus \mathbb{Q}$. Notice that, for any $l_1,l_2 \in \mathbb{Z}, w^{l_1}=w^{l_2}$ if and only if $l_1=l_2$. Let $\mu$ be the unique infinite path in $(w,0)\Lambda_\bn^\infty$ such that $\mu(p,q)=(w^{\bn^p},q-p)$ for all $p \leq q \in \mathbb{N}^k$. For $p \neq q \in \mathbb{N}^k$, since $\bn^p \neq \bn^q$,
we have $\sigma^{p}(\mu)(0,0) \neq \sigma^{q}(\mu)(0,0)$ and so  $\sigma^{p}(\mu) \neq \sigma^{q}(\mu)$. Therefore $\Lambda_\bn$ satisfies Condition (A).
\end{proof}


\begin{lem}\label{approximate lemma}
Let $\mathbb{F}_\theta^+ \bowtie \mathbb{Z}$ and $\mathbb{F}_\alpha^+ \bowtie \mathbb{Z}$ be two standard product of the odometers
$\{(\mathbb{Z},[n_i])\}_{i=1}^{\fk}$ and $\{(\mathbb{Z},[m_j]\}_{j=1}^\ell$, respectively. Suppose that $1 \leq \fk \leq \ell \leq \infty$ and that $n_i=m_i$ for all $i=1,\dots,\fk$. Then there is a unital embedding from $\mathcal{O}_{X(\Lambda_\bn)}$ into $\mathcal{O}_{X(\Lambda_\bm)}$. Hence there exists a unital embedding from $\mathcal{Q}(\mathbb{F}_\theta^+ \bowtie \mathbb{Z})$ into $\mathcal{Q}(\mathbb{F}_\alpha^+ \bowtie \mathbb{Z})$.
\end{lem}

\begin{proof}
Denote by $\imath:X(\Lambda_\bn) \to \mathcal{O}_{X(\Lambda_\bn)}$ and $\jmath:X(\Lambda_\bm) \to \mathcal{O}_{X(\Lambda_\bm)}$ the universal Cuntz-Pimsner covariant representations of $X(\Lambda_\bn)$ and $X(\Lambda_\bm)$, respectively. We realize $\mathbb{N}^\fk$ as a subsemigroup of $\mathbb{N}^{\ell}$ by $p \mapsto (p,0)$. For $p \in \mathbb{N}^\fk$,
we also realize $X(\Lambda_\bn)_p$ as $X(\Lambda_\bm)_p$ as they are isomorphic as C*-correspondences over $\rC(\mathbb{T})$. For $p \in \mathbb{N}^\fk$, define $\psi_p:X(\Lambda_\bm) \to \mathcal{O}_{X(\Lambda_\bm)}$ to be $\jmath_p$. By piecing $\{\psi_p\}_{p \in \mathbb{N}^\fk}$ together, one obtains a Cuntz-Pimsner covariant representation of $X(\Lambda_\bn)$. Let $h:\mathcal{O}_{X(\Lambda_\bn)} \to \mathcal{O}_{X(\Lambda_\bm)}$ be the unital homomorphism induced from the universal property of $\mathcal{O}_{X(\Lambda_\bn)}$. By \cite[Theorem~4.1, Corollary~5.2]{SY10}, $h \vert_{\imath_0(X(\Lambda_\bn))}$ is injective. Let $\gamma$ be the gauge action for $\jmath$. Then by restriction $\gamma$ induces a gauge action for $\psi$.
Invoking Theorem~\ref{gauge-inv uni thm} yields that $h$ is injective. The second statement follows from Theorem~\ref{T:top2}.
\end{proof}

\begin{thm}
\label{T:simple}
Let $\Lambda_\bn$ be the topological $k$-graph constructed in Definition \ref{D:top2}.
\begin{itemize}
\item[(i)]
$\O_{X(\Lambda_\bn)}$ is nuclear.
\item[(ii)]
Then $\mathcal{O}_{X(\Lambda_\bn)}$ is simple if and only if $\{\ln {n_i}\}_{1\le i\le k}$ is rationally independent.
\item[(iii)]
If $\{\ln {n_i}\}_{1\le i\le k}$ is rationally independent, then $\mathcal{O}_{X(\Lambda_\bn)}$ is purely infinite.
\item[(iv)]
$\O_{X(\Lambda_\bn)}$ is a unital UCT Kirchberg algebra if and only if $\{\ln {n_i}\}_{1\le i\le k}$ is rationally independent.
\end{itemize}
\end{thm}

\begin{proof}
(i) First of all, suppose $k \neq \infty$.
From the proof of Theorem \ref{T:4.7},  we obtain that $\ca(\G_{\Lambda})$ is nuclear and
$\ca(\G_{\Lambda})\cong \O_{X(\Lambda)}$.
Therefore $ \O_{X(\Lambda_\bn)}$ is nuclear.

Now suppose that $k=\infty$. By Lemma~\ref{approximate lemma}, we obtain an increasing sequence $\{\A_i:=\O_{X(\Lambda_{(n_1,\dots,n_i)})}\}_{i=1}^{\infty}$ of unital C*-subalgebras of $\O_{X(\Lambda_\bn)}$ such that $\bigcup_{i=1}^{\infty}\A_i$ is dense in $\O_{X(\Lambda_\bn)}$. Since each $\A_i$ is nuclear by the preceding paragraph, we deduce
that $\O_{X(\Lambda_\bn)}$ is nuclear.

(ii) The proof of ``If": Suppose that $k \neq \infty$. By Lemma \ref{L:aper}, $\Lambda_\bn$ satisfies Condition (A).

Fix $(z,0) \in \Lambda_\bn^0$, and let $\mu$ be the unique infinite path in $(z,0) \Lambda_\bn^\infty$. For $p \in \mathbb{N}^k$ and $w \in \mathbb{T}$ such that $w^{\bn^p}=z$, we have $r(w,p)=(w,0), s(w,p)=(w^{\bn^p},0)=(z,0)$. So $(w,0) \in \mathrm{Orb}^+((z,0))$.
Let $(z',0) \in \Lambda_\bn^0$ and $\epsilon>0$. Then we can always find $p \in \mathbb{N}^k$ with $\bn^p$ large enough so that the distance between $z'$ and
one of $\bn^p$-th roots of $z$
is less than $\epsilon$. Hence $\mathrm{Orb}^+((z,0))$ is dense in $\Lambda_\bn^0$. Since $\mathrm{Orb}^+((z,0)) \subset \mathrm{Orb}((z,0),\mu)$,
clearly $\mathrm{Orb}((z,0),\mu)$ is dense in $\Lambda_\bn^0$ as well. Therefore by Theorem~\ref{T:4.7}, $\mathcal{O}_{X(\Lambda_\bn)}$ is simple.

Now suppose that $k=\infty$. By Lemma~\ref{approximate lemma}, we obtain an increasing sequence $\{\A_i:=\O_{X(\Lambda_{(n_1,\dots,n_i)})}\}_{i=1}^{\infty}$ of unital C*-subalgebras of $\O_{X(\Lambda_\bn)}$ such that $\bigcup_{i=1}^{\infty}\A_i$ is dense in $\O_{X(\Lambda_\bn)}$. Since each $\A_i$ is simple by the above argument, we deduce
that $\O_{X(\Lambda_\bn)}$ is simple.

The proof of ``Only if": We must show that the rational dependence of $\{\ln {n_i}\}_{1\le i\le k}$ implies that $\O_{X(\Lambda_\bn)}$ is not simple, equivalently $\Q(\Fth\bowtie\bZ)$
is not simple by Theorem \ref{T:top2}. Now suppose that
$\{\ln {n_i}\}_{1\le i\le k}$ is rationally dependent.
Then there exist $p \neq q \in \mathbb{N}^k$ such that $\bn^p=\bn^q$. Let $A:=\{1 \leq i \leq k:p_i<q_i\}$ and $B:=\{1 \leq i \leq k:p_i>q_i\}$. We may assume that $A \neq \mt$. Then $\Pi_{i \in A} n_i^{q_i-p_i}=\Pi_{i \in B} n_i^{p_i-q_i}$. Inspired by \cite{Cun08}, in what follows, we construct a representation of $\Q(\Fth\bowtie \bZ)$ on $\ell^2(\mathbb{Z})$. To this end, let $\{\delta_m\}_{m \in \mathbb{Z}}$ denote the standard orthonormal basis of $\ell^2(\mathbb{Z})$. Define
\begin{align*}
F(\delta_m)&:=\delta_{m+1}\ (m\in\bZ); \\
G_{x^i_\fs}(\delta_m)&:=\delta_{\fs+n_im} \ (m\in\bZ, \fs\in [n_i], 1\le i\le k).
\end{align*}
Then $F$ is a unitary and $G_{x^i_\fs}$'s  are isometries. Some calculations show that $\{F, G_{x^i_\fs}: \fs\in[n_i],  1\le i\le k\}$ satisfy Conditions~(i)--(iii)
of Theorem \ref{T:universal}. By Theorem \ref{T:universal}, there exists a nonzero homomorphism $\pi: \mathcal{Q}(\bF_\theta^+ \bowtie \mathbb{Z}) \to B(\ell^2(\mathbb{Z}))$ such that
$\pi(f)=F, \pi(g_{x^i_\fs})=G_{x^i_\fs}$ for all $\fs\in[n_i], 1\le i\le k$. Since $\prod_{i \in A} n_i^{q_i-p_i}=\prod_{i \in B} n_i^{p_i-q_i}$,  one has
$\prod_{i \in A} G_{x^i_0}^{q_i-p_i}=\prod_{i \in B} G_{x^i_0}^{p_i-q_i}$. Suppose that
$\prod_{i \in A} g_{x^i_0}^{q_i-p_i}=\prod_{i \in B} g_{x^i_0}^{p_i-q_i}$ for a contradiction.
Let $\gamma:\prod_{i=1}^{k}\bT \to \Aut(\Q(\Fth\bowtie \bZ))$ be the gauge action induced from the universal property of $\mathcal{O}_{X(\Lambda_\bn)}$. By Theorem \ref{T:top2},
\begin{align*}
0=\gamma_{z}\left(\prod_{i \in A} g_{x^i_0}^{q_i-p_i}-\prod_{i \in B} g_{x^i_0}^{p_i-q_i}\right)=\prod_{i \in A}z_i^{q_i-p_i}\prod_{i \in A} g_{x^i_0}^{q_i-p_i}-\prod_{i \in B} g_{x^i_0}^{p_i-q_i}
\end{align*}
for all $z \in \prod_{i=1}^{k}\bT$ such that $z_i=1$ whenever $i \in B$. Since $p_i<q_i$ for $i \in A \neq \emptyset$, we deduce that $\prod_{i \in A} g_{x^i_0}^{q_i-p_i}=0$ which is impossible.
So $0\ne \prod_{i\in A} g_{x^i_0}^{q_i-p_i}-\Pi_{i \in B} g_{x^i_0}^{p_i-q_i}\in \ker(\pi)$.
Thus $\ker(\pi)$ is a nontrivial closed two-sided ideal in $\mathcal{Q}(\bF_\theta^+ \bowtie \mathbb{Z})$, implying that $\mathcal{Q}(\bF_\theta^+ \bowtie \mathbb{Z})$ is not simple.

(iii) Suppose that $k<\infty$. Then $\mathcal{O}_{X(\Lambda_\bn)}$ is simple from (ii) above. As shown in (ii), $\mathrm{Orb}^+((1,0))$ is dense in $\Lambda_\bn^0$. Fix an open neighborhood $U$ of $1$. Then there exists $\delta>0$ such that $n_1\delta$ is small enough (say $<1/4$) and $V:=\{e^{2\pi i \theta}:\theta \in (-\delta,\delta)\} \subset U$. Denote by $U_1:=V \times \{e_1\}$. It is straightforward to see that $r(U_1) \subset V \times \{0\}$ and $\overline{V \times \{0\}} \subsetneq s(U_1)$. So $V \times \{0\}$ is contracting. Hence $\Lambda_\bn$ is contracting. By Theorem~\ref{T:4.7}, $\mathcal{O}_{X(\Lambda_\bn)}$ is purely infinite.

Now suppose that $k=\infty$. By Lemma~\ref{approximate lemma}, we obtain an increasing sequence $\{\A_i:=\O_{X(\Lambda_{(n_1,\dots,n_i)})}\}_{i=1}^{\infty}$ of unital C*-subalgebras of $\O_{X(\Lambda_\bn)}$ such that $\bigcup_{i=1}^{\infty}\A_i$ is dense in $\O_{X(\Lambda_\bn)}$. Since each $\A_i$ is simple and purely infinite by the above paragraph, we deduce
that $\O_{X(\Lambda_\bn)}$ is simple and purely infinite.

(iv) Since $\G_\Lambda$ is amenable and $\O_{X(\Lambda)}\cong \ca(\G_\Lambda)$,  $\O_{X(\Lambda)}$ satisfies the UCT due to \cite{Tu99}. The rest of (iv) now easily follows from (i)-(iii). 
\end{proof}

As an immediate consequence of Theorems \ref{T:top2} and \ref{T:simple}, one has

\begin{thm}\label{T:simplepure}
Let $\Fth\bowtie\bZ$ be the standard product of odometers $\{(\bZ$, $[n_i])\}_{=1}^k$. Then
\begin{itemize}
\item[(i)]
$\Q(\Fth\bowtie \bZ)$ is nuclear;
\item[(ii)]
$\Q(\Fth\bowtie \bZ)$ is a unital UCT Kirchberg algebra $\Leftrightarrow$ $\{\ln {n_i}\}_{1\le i\le k}$ is rationally independent $\Leftrightarrow$ $\O_\theta$ is simple
$\Leftrightarrow$ $\Fth$ is aperiodic.
\end{itemize}
\end{thm}

\begin{proof}
By Theorems \ref{T:top2} and \ref{T:simple}, it remains to show that $\Fth$ is aperiodic  $\Leftrightarrow$ $\O_\theta$ is simple
$\Leftrightarrow$ $\{\ln {n_i}\}_{1\le i\le k}$ is rationally independent.

$\Fth$ is aperiodic $\Rightarrow\O_\theta$ is simple: If $k \neq \infty$ then this follows from \cite[Corollary 8.6]{DY09}. If $k=\infty$, then there is an increasing sequence $\{\A_i\}_{i=1}^{\infty}$ of C*-subalgebras of $\O_\theta$ such that each $\A_i$ is the C*-algebra of an aperiodic single-vertex finite-rank graph and the union of $\{\A_i\}_{i=1}^{\infty}$ is dense in $\O_\theta$. So $\O_\theta$ is simple.

$\O_\theta$ is simple $\Rightarrow \{\ln {n_i}\}_{1\le i\le k}$ is rationally independent: We prove its contraposition. Suppose that  $\{\ln {n_i}\}_{1\le i\le k}$ is rationally dependent. Notice that $\{g_{x^i_\fs}: \fs\in[n_i], 1\le i\le k\}$ is a Cuntz-Krieger $\bF_\theta^+$-family in $\mathcal{Q}(\Fth\bowtie \bZ)$. Then there is a homomorphism $\rho:\mathcal{O}_{\theta} \to \mathcal{Q}(\bF_\theta^+ \bowtie \mathbb{Z})$ induced from the universal property of $\mathcal{O}_{\theta}$. Let $\pi: \mathcal{Q}(\bF_\theta^+ \bowtie \mathbb{Z}) \to B(\ell^2(\mathbb{Z}))$ be the nonzero homomorphism given in the proof of Theorem~\ref{T:simple}. Since the kernel of $\pi \circ \rho$ is a nontrivial closed two-sided ideal of $\O_\theta, \O_\theta$ is not simple.

$\{\ln {n_i}\}_{1\le i\le k}$ is rationally independent $\Rightarrow\Fth$ is aperiodic: If $k \neq \infty$ then this follows from \cite[Theorem 3.4, Definition 3.6 and Theorem 7.1]{DY09}. If $k=\infty$ and if $\Fth$ is periodic, then there exists $1 \leq l <\infty$ such that the $l$-graph $\bF_\alpha^+$ determined by $\{n_i\}_{i=1}^{l}$ is periodic as well (\cite{Yan15}).
So $\{\ln {n_i}\}_{1\le i\le l}$ is rationally dependent.  Hence $\{\ln {n_i}\}_{1\le i\le k}$ is also rationally dependent. Therefore we are done.
\end{proof}

As an immediate consequence of Theorem \ref{T:simplepure}, the boundary quotient C*-algebra left in \cite{BRRW14} is now well-understood.

\begin{cor}
Let $\Fth\bowtie \bZ$ be the standard product of $2$-odometers over $n$-letter and $m$-letter alphabets with
$\gcd(n,m)=1$.
Then $\Q(\Fth\bowtie \bZ)$ is a unital UCT Kirchberg algebra.
\end{cor}

\begin{rem}
After we finished this paper, Nicolai Stammeier informed us that, using a completely different approach, it was shown in \cite{BOS15} that $\Q(\Fth\bowtie\bZ)$ is 
nuclear, simple and purely infinite if $\{n_i\}_{i=1}^k \subset \mathbb{N} \setminus \{0,1\}$ is a pairwise coprime set. So Theorem \ref{T:simplepure} generalizes their results. Furthermore, we exhibit several sufficient and necessary conditions for the simplicity of $\Q(\Fth\bowtie\bZ)$.
\end{rem}

As an extreme case, let $k=1$ in Theorem \ref{T:simplepure}. Then we obtain that the boundary quotient C*-algebra $\Q(\Fn\bowtie \bZ)$
of the odometer action on a $n$-letter alphabet with $n\ge 2$ is nuclear, simple and purely infinite.
Recall from Theorem \ref{T:universal} that $\Q(\Fn\bowtie \bZ)$ is the universal C*-algebra
generated by a unitary $f$ and $n$ isometries $g_{x_i}$ $(i\in[n])$ such that
\begin{itemize}
\item[(i)]
$\sum_{i\in[n]}g_{x_i}g_{x_i}^*=1$;
\item[(ii)]
$f g_{x_{i}}= \begin{cases}
    g_{x_{i+1}} &\text{ if } 0 \leq i <n-1 \\
    g_{x_{0}} f &\text{ if } i=n-1.
\end{cases}
$
\end{itemize}

Also, given $n\ge 2$, the \textit{$n$-adic ring C*-algebra $\Q_n$ of the integers} is the universal C*-algebra generated by a unitary $u$ and an isometry $s$ satisfying
\begin{align}
\label{E:Qn}
\sum_{i\in [n]}u^is(u^i s)^*=1\quad \text{and}\quad  u^ns=su.
\end{align}

\begin{cor}
\label{C:Qn}
There is an isomorphism $\pi:  \Q(\Fn\bowtie\bZ)\to \Q_n$ such that $\pi(f)=u$ and $\pi(g_{x_i})=u^is$ for all $i\in[n]$.
\end{cor}

\begin{proof}
From \eqref{E:Qn} it is easy to check that $\{\pi(f), \pi(g_{x_i}):i\in[n]\}$ satisfies Conditions (i) and (ii) above. So by the universal property of
$\Q(\Fn\bowtie\bZ)$, $\pi$ can be extended to isomorphism as $\Q(\Fn\bowtie\bZ)$ is simple by Theorem \ref{T:simplepure}.
\end{proof}

It turns out that $\Q_n$ is isomorphic to the graph C*-algebra $\O(E_{n,1})$ of the topological graph $E_{n,1}$ of Katsura studied in \cite{Kat08},
where it is shown that $\Q_n$ is nuclear, simple and purely infinite. So we recover this result here.
Also let us remark that $\Q_2$ was systematically studied by Larsen-Li in \cite{LL12}.

\begin{eg}
Consider the standard product $\Fth\bowtie \bZ$ of the odometers $(\bZ,[n_i])$ with $n_i=n$ for all $1\le i\le k<\infty$.
To make our example more interesting, let $n\ge 2$ and $k\ge 2$.
This is a special case of Example \ref{Eg:nrLCM} with $G=\bZ$ and the self-similar action being the odometer.
\end{eg}

By Theorem \ref{T:simplepure}, $\Q(\Fth\bowtie\bZ)$ is not simple.
Actually one can apply some results in \cite{DY09} to obtain its structure as follows.

\begin{obs}
Keep the above notation. One has
\[
\Q(\Fth\bowtie \bZ)\cong \Q_n\otimes \rC(\bT^{k-1}).
\]
\end{obs}

\begin{proof}
Notice from \cite[Corollary 7.2, Theorem 8.5 and Corollary 8.7]{DY09} that $\O_\theta\cong\O_n\otimes\ca(\{(e_0^1)^*e_0^i: 1\le i\le k-1\})\cong \O_n\otimes \rC(\bT^{k-1})$.
Note
\[
s_at_u=t_{a\cdot u}s_{a|_u}\Rightarrow t_u s_{a|_u}^*=s_a^*t_{a\cdot u}\Rightarrow s_{a|_u}t_u^*=t_{a\cdot u}^*s_a.
\]
 Then we compute
 \begin{align*}
 (e_0^1)^*e_0^i s_1
 &=(e_0^1)^* e_{1\cdot n-1}^i s_{1|_{n-1}}
   =(e_0^1)^* s_1 e_{n-1}^i
   =(e_{1\cdot (n-1)}^1)^* s_1 e_{n-1}^i\\
 &=s_{1|_{n-1}}(e_{n-1}^1)^* e_{n-1}^i
   =s_{1}(e_{n-1}^1)^* e_{n-1}^i
   =s_{1}(e_{0}^1)^* e_{0}^i
 \end{align*}
 for all $1\le i\le k-1$. Thus
 \[
\Q(\Fth\bowtie \bZ)\cong \Q(\Fn \bowtie \bZ)\otimes \rC(\bT^{k-1}) \cong \Q_n\otimes \rC(\bT^{k-1})
\]
by Corollary \ref{C:Qn}.
\end{proof}

\subsection{Relations between $\mathcal{Q}(\bF_\theta^+ \bowtie \mathbb{Z})$ and $\mathcal{Q}_\mathbb{N}$}
\label{SS:relQN}

Cuntz in \cite[Definition~3.1]{Cun08} defined $\mathcal{Q}_{\mathbb{N}}$ to be the universal C*-algebra generated by a unitary $u$ and a family of isometries $\{s_n\}_{n \in \mathbb{N}^\times}$ satisfying
\[
s_n s_m=s_{nm},\ u^n s_n=s_n u,\ \sum_{t=0}^{n-1}u^{t}s_n s_n^* u^{-t}=1\qforal n,m \in \mathbb{N}^\times.
\]

In what follows, we discuss some relations between $\mathcal{Q}_\mathbb{N}$ and the boundary quotient C*-algebra $\mathcal{Q}(\bF_\theta^+ \bowtie \mathbb{Z})$ of the standard product of odometers $\{(\bZ, [n_i])\}_{i=1}^k$. For this, define
\[
F:=u, \quad G_{x^i_{\ft}}:=u^\ft s_{n_i} \quad (\ft\in[n_i],\ 1\le i\le k).
\]
A simple calculation shows that $\{F, G_{x^i_{\ft}}:\ft\in[n_i], 1 \leq i \leq k\}$ satisfy Conditions (i)--(iii) of Theorem \ref{T:universal}. By
Theorem \ref{T:universal}, there exists a homomorphism
\begin{align}
\label{E:rho}
\rho: \mathcal{Q}(\bF_\theta^+ \bowtie \mathbb{Z}) \to \mathcal{Q}_{\mathbb{N}}
\end{align}
such that
$\rho(f)=F$ and $\rho(g_{x^i_{\ft}})=G_{x^i_{\ft}}$ for all $\ft\in[n_i]$, $1 \leq i \leq k$.

If $k=\infty$ and $\{n_i\}_{i=1}^{\infty}$ is the set of all prime numbers, then $\rho$ is an isomorphism by Theorem~\ref{T:simplepure}. Thus
one has

\begin{cor}
\label{C:rho}
If $k=\infty$ and $\{n_i\}_{i=1}^{\infty}$ is the set of all prime numbers, then $\Q(\Fth\bowtie\bZ)\cong\Q_\bN$.
\end{cor}

Let us finish this paper by characterizing when the above homomorphism $\rho$ is injective.

\begin{thm}
The homomorphism $\rho$ in \eqref{E:rho} is injective if and only if $\{\ln {n_i}\}_{1\le i\le k}$ is rationally independent.
\end{thm}
\begin{proof}
If $\{\ln {n_i}\}_{1\le i\le k}$ is rationally independent, then $\mathcal{Q}(\bF_\theta^+ \bowtie \mathbb{Z})$ is simple by Theorem \ref{T:simplepure}. So $\rho$ is injective.

Conversely, suppose that $\{\ln {n_i}\}_{1\le i\le k}$ is rationally dependent. Then there exist $p \neq q \in \mathbb{N}^k$ such that $\bn^p=\bn^q$. It is straightforward to see that
$\rho\Big(\prod_{i=1}^{k}g_{x_{0}^{i}}^{p_i}\Big)=\rho\Big(\prod_{i=1}^{k}g_{x_{0}^{i}}^{q_i}\Big)$. Hence $\rho$ is not injective.

\end{proof}

\subsection*{Acknowledgments}

The first author wants to thank the second author for being generous in sharing ideas and for numerous valuable discussions. Both authors are very grateful to
Dr. Nicolai Stammeier for providing us with valuable comments, pointing out an error in Subsection \ref{SS:relQN} of an early version of this paper, and bringing the references \cite{ABLS16, BOS15, Stam16} to our attention.

\end{document}